\documentclass{article}
\usepackage{graphicx,amsthm,amsmath,amsfonts,amssymb,comment,enumitem,color,url,hyperref,rotating} 

\newtheorem{theorem}{Theorem}
\newtheorem{definition}{Definition}
\newtheorem{corollary}{Corollary}
\newtheorem{remark}{Remark}
\newtheorem{proposition}{Proposition}

\DeclareMathOperator{\EE}{\mathbb{E}}
\DeclareMathOperator{\VV}{\mathbb{V}}
\DeclareMathOperator{\Cov}{\text{Cov}}

\title{Data Reuse and the Long Shadow of Error: Splitting, Subsampling, and Prospectively Managing Inferential Errors}
\author{Reid Dale$^1$, Jordan Rodu$^2$, Maria E. Currie$^3$, Mike Baiocchi$^4$}
\date{
	$^1$Stanford University School of Medicine Department of Cardiothoracic Surgery \\ \texttt{reiddale@stanford.edu}\\%
    $^2$University of Virginia Department of Statistics\\
    \texttt{jsr6q@virginia.edu}\\
    $^3$Stanford University School of Medicine Department of Cardiothoracic Surgery \\
    \texttt{mecurrie@stanford.edu}\\
     $^4$Stanford University Department of Epidemiology $\&$ Population Health \\ \texttt{baiocchi@stanford.edu}\\
	\today
}

\begin{document}

\maketitle

\begin{abstract}
When multiple investigators analyze a common dataset, the data reuse induces dependence across testing procedures, affecting the distribution of errors. Existing techniques of managing dependent tests require either cross-study coordination or post-hoc correction. These methods do not apply to the current practice of uncoordinated groups of researchers independently evaluating hypotheses on a shared dataset. We investigate the prospects of using subsampling techniques implemented at the level of individual investigators to remedy dependence with minimal coordination.

To this end, we establish the asymptotic joint normality of test statistics for the class of asymptotically linear test statistics, decomposing the covariance matrix as the product of a data overlap term and a test statistic association term. This decomposition shows that controlling data overlap is sufficient to control dependence, which we formalize through the notion of Expected Variance Ratio.

This enables the closed form derivation of the variance of the joint rejection region under the global null as a function of pairwise correlations of test statistics. We adopt mean-variance portfolio theory to measure risk, defining the Expected Variance Ratio (EVR) as the ratio of the expected variance of the Type I error count to the independent baseline. Familywise error rate is demonstrated to be minimized precisely when this variance is maximized.

We show that data splitting is asymptotically optimal among rules that ensure exact independence, but require coordination. We then use concentration inequalities to establish that subsampling techniques implementable by individual investigators can ensure an EVR close to $1$. 

Finally, we show that such subsampling techniques are able to simultaneously perform a number of tests while ensuring sufficient power and that the bounded EVR is $O\left(\frac{1}{r^2}\right)$ compared to data splitting's $O\left(\frac{1}{r}\right)$, where $r$ is the per-statistic fraction of data required.
\end{abstract}

\newpage

\section{Introduction}

Large-scale registries and public-use datasets have enabled investigators to conduct observational studies at unprecedented scale and efficiency. Common practice is for each investigator to use all observations meeting the inclusion criteria of their study. But when many investigators draw on the same dataset, their test statistics become dependent through overlapping data, and this dependence has consequences for the collective reliability of the resulting body of evidence. This problem is likely to become more acute as AI-assisted scientific workflows increase both the volume and pace at which analyses of shared datasets are produced.

The recognition that data reuse induces dependent testing is not new; recently, this phenomenon has been termed ``dataset decay'' by Thompson et al. \cite{thompson_dataset_2020}. A substantial body of work addresses this dependence in settings where it can be managed through centralized coordination or post-hoc correction. \textit{Within-study} methods include multiple testing adjustments  \cite{benjamini_controlling_1995,holm1979simple}, seemingly unrelated regressions \cite{zellner_efficient_1962}, and sequential testing \cite{johari_always_2019}. These methods handle dependence among tests conducted by a single investigator with knowledge of all analyses being performed. \textit{Centralized} approaches such as data splitting \cite{cox_note_1975}, $\alpha$-spend \cite{demets1994interim}, $\alpha$-investing \cite{foster_investing_2008}, and
the adaptive inference framework of Dwork et al. \cite{dwork_preserving_2015}, require either a data splitting authority or a mechanism for coordinating queries across analysts. \textit{Post hoc} corrections for dependent tests have been developed in the meta-analysis literature \cite{cheung2019guide}, in the context of reused natural experiments \cite{heath_reusing_2023}, and through post-hoc estimates of the false discovery rate \cite{efron_large-scale_2010,efron2010correlated}. None of these approaches, however, are designed for the setting we consider: the \textit{prospective} management of correlation among a large number of uncoordinated investigators, each independently selecting hypotheses and test statistics to evaluate on a common dataset, with no central catalogue of which observations have been used and no requirement that analysts coordinate their designs.

We approach the problem from the perspective of mean-variance portfolio theory \cite{joshi_introduction_2013,markowitz_portfolio_1952}, treating the indicator functions of rejection regions $\mathbf{1}_{R_i}$ of $C$ hypothesis tests as assets in a portfolio and measuring risk by the variance of the total error count $E = \sum_{i=1}^C \mathbf{1}_{R_i}$ under the global null. The key observation motivating this framing is that data reuse does not affect the \textit{expected} number of Type I errors. By linearity of expectation, $\mathbb{E}[E] = C\alpha$ regardless of the dependence structure, but the variance of the error distribution is affected. Independent tests yield $\mathbb{V}(E) = C\alpha(1-\alpha)$, growing linearly in $C$; dependent tests can yield variance growing quadratically in $C$, concentrating probability mass on the event of a catastrophically large number of simultaneous errors. We formalize this via the \textit{Expected Variance Ratio} (EVR), defined as $\text{EVR}(S) = \mathbb{E}_S[\mathbb{V}(E)] / C\alpha(1-\alpha)$, which compares the expected variance under a data allocation procedure $S$ to the independent baseline.

We show, perhaps counterintuitively, that familywise error rate (FWER) is minimized precisely when variance is \textit{maximized}: perfectly correlated tests achieve $\text{FWER} = \alpha$ but $\mathbb{V}(E) = C^2\alpha(1-\alpha)$. This suggests that FWER is the wrong criterion for managing large-scale epistemic risk, as it rewards the concentration of errors.

Our analysis proceeds as follows. We first establish the dependence structure of test statistics computed on overlapping subsets of a shared dataset. Theorem~\ref{thm:asymptotic_linearity} shows that for the broad class of asymptotically linear estimators (including $M$-estimators, $U$-statistics, linear rank statistics, Kaplan-Meier estimators, and Cox regression) the joint distribution of standardized test statistics is asymptotically multivariate normal with covariance $\Sigma_{ij} = \varrho_{ij}\, \mathbb{E}[\psi_i(X)\psi_j(X)]$, decomposing cleanly into an \textit{overlap} component $\varrho_{ij}$ determined by the fraction of shared data and a \textit{statistical association} component $\mathbb{E}[\psi_i(X)\psi_j(X)]$ determined by the relationship between the influence functions of the two estimators. Corollary~\ref{cor:pearson_cor} shows that controlling data overlap is sufficient to control dependence, since the asymptotic Pearson correlation is bounded above by $\varrho_{ij}$.

Under the resulting joint normality, we derive closed-form expressions for the variance of the error count as a function of pairwise correlations (Proposition~\ref{prop:correlated_rejection}) and bound the excess variance contributed by correlation (Corollary~\ref{cor:bounded_correlation_variance}). The excess variance is a sum over pairs of a function $R(\rho_{ij}, c_{\alpha/2})$ that is monotonically increasing in $|\rho_{ij}|$ and subquadratic, growing as $\alpha(1-\alpha)\rho^2$ near the origin.

We then turn to the design of data allocation procedures. In Section~\ref{sec:subsampling}, we show that partitioning the dataset into $C$ disjoint subsets of equal size guarantees independence across tests and achieves the maximin asymptotic relative efficiency of $1/C$ among all partitioning procedures (Proposition~\ref{prop:asymptotic_optimality_of_splitting}). However, data splitting requires centralized coordination: a registry of which observations have been assigned to which study, ensuring no overlaps occur. This makes it ill-suited for the federated setting.

We therefore consider a class of procedures we call \textit{egalitarian subsampling}, in which each investigator independently draws a uniform random subsample of size $r(N) \cdot N$ from the shared dataset (Definition~\ref{def:subsampling_families}). These procedures are implementable without any coordination across investigators. While they do not guarantee independence and are indeed asymptotically suboptimal in terms of relative efficiency when $r(N) \to 0$ (Proposition~\ref{prop:subsampling_suboptimal}), they can nevertheless guarantee bounded $EVR$. Using tail bounds for hypergeometric random variables (Proposition~\ref{prop:subsampling_tail_bound}), we show that the probability of large pairwise correlations decays exponentially in the sample size, allowing us to bound the expected variance (Proposition~\ref{cor:variance_under_random_subsampling}) and hence the EVR (Corollary~\ref{cor:evr_bound}). The bound decomposes into a quadratic term $R(\rho_0, c_{\alpha/2})$ depending on the expected magnitude of pairwise correlations and an exponentially decaying tail contribution from the probability that any pair of tests has an unusually large overlap.

These bounds allow us to characterize the \textit{capacity} of a dataset, which we define as the number of studies $C$ it can sustain at a given EVR tolerance $1+\delta$. Under egalitarian subsampling with fraction $r(N) = b/\sqrt{N}$, Proposition~\ref{prop:egalitarian_evr_inversion} shows that the capacity scales as $O\left(\frac{1}{r(N)^2}\right) = O(N/b^2)$, compared to $O(N/M) = O\left(\frac{1}{r(N)} \right)$ for data splitting (where $M$ is the per-statistic sample size requirement). Since each investigator uses only $b\sqrt{N}$ observations (far less than the full dataset, but enough for adequate power for large $N$) the dataset can support a substantially larger number of analyses under the relaxed requirement of bounded EVR compared to data splitting.

We demonstrate these results in two worked examples in Section~\ref{sec:worked_example}. In Section~\ref{subsec:reused_control}, we revisit the classical problem of reused control groups \cite{dunnett1955multiple}, showing that egalitarian subsampling with $b = 10$ can sustain over three times as many pairwise treatment-control contrasts as data splitting while keeping EVR below $1.1$. In Section~\ref{subsec:sur}, we evaluate the framework in the setting of families of seemingly unrelated regressions under varying correlation structures, confirming that egalitarian subsampling maintains near-independent error variance even when the underlying covariates are highly correlated.

\section{Dependence Structure of Asymptotically Linear Test Statistics Under Data Reuse}

In this section we observe that a wide class of test statistics yield asymptotically multivariate normal distributions. Our result will have two key implications:

\begin{enumerate}
    \item By virtue of being multivariate normal, \textit{asymptotically} the pairwise correlations between test statistics are sufficient to specify the dependence structure on test statistics, and 
    \item The explicit formula for the covariance matrix illustrates exactly how dependence scales with the overlap of data among the constituent test statistics.
\end{enumerate}

Recall the definition of Asymptotic linearity (e.g. section 25.9 of van der Vaart)

\begin{definition}\label{def:asymptotic-linearity}
A sequence of statistics $T^n = T^n(X_1, \ldots, X_n)$ estimating a parameter $\theta$ is \emph{asymptotically linear} with influence function $\psi$ if
\[
T^n - \theta = \frac{1}{n}\sum_{i=1}^n \psi(X_i) + o_p(n^{-1/2})
\]
where $\mathbb{E}[\psi(X)] = 0$ and $\VV[\psi(X)] < \infty$ for each distribution in the family $\Theta$.\footnote{In this manuscript we need only assume that this holds in a local neighborhood of the null parameter $\theta_0 \in \Theta$.}
\end{definition}

A wide class of test statistics commonly used in applied settings by investigators are asymptotically linear including: sample means, U-statistics (\cite{van2000asymptotic}, Proofs of Theorems 12.3 and 12.6) , $M$-estimators (\cite{van2000asymptotic}, Proofs in Section 5.3), permutation tests and linear rank statistics (\cite{van2000asymptotic} Proofs in  Section 13.5), Kaplan-Meier estimators \cite{cai1998asymptotic}, and the Cox proportional hazards model. \cite{chen_iidcox}. By contrast, $\chi^2$ test statistics are not asymptotically linear.

Crucially, asymptotically linear test statistics on overlapping data admit a central limit theorem.

\begin{theorem}\label{thm:asymptotic_linearity}
For each $N \geq 1$, let $X_1, \ldots, X_N$ be iid. Let $C \in \mathbb{N}^{>0}$ be the number of test statistics. For $i = 1, \ldots, C$, let $D_i^{(N)} \subseteq [N]$ with $|D_i^{(N)}| = n_i^{(N)}$, and let
\[
T_i^{(N)} = T_i\!\left((X_j)_{j \in D_i^{(N)}}\right)
\]
be a statistic computed on sample $D_i^{(N)}$, estimating a parameter $\theta_i$, and asymptotically linear with influence function $\psi_i$:

Define $r_i(N) = \frac{n_i^{(N)}}{N}$ and $\omega_{ij}(N) = \frac{|D_i^{(N)} \cap D_j^{(N)}|}{N}$. As $N \to \infty$, assume that $Nr_{i}(N) \to \infty$ and that 
\begin{align} \frac{\omega_{ij}(N)}{\sqrt{r_i(N)r_j(N)}}  \rightarrow \varrho_{ij}
\end{align}
Define the standardized statistics
\[
S_i^{(N)} = \sqrt{n_i^{(N)}}\bigl(T_i^{(N)} - \theta_i\bigr).
\]

Then:
\[
\bigl(S_1^{(N)}, \ldots, S_C^{(N)}\bigr) \;\xrightarrow{d}\; \mathcal{N}_C(0, \Sigma)
\]
where $\Sigma \in \mathbb{R}^{C \times C}$ is the matrix with entries
\begin{equation}\label{eq:sigma_entries}
\Sigma_{ij} = \varrho_{ij} \mathbb{E}[\psi_i(X)\psi_j(X)]  \text{for all } 1 \leq i, j \leq C.
\end{equation}
\end{theorem}

The proof of this theorem is in Appendix \ref{proof:proof_of_thm:asymptotic_linearity} 

\begin{remark}
The components of the covariance matrix $\Sigma_{ij}$ have an important interpretation. The formula $\Sigma_{ij} = \varrho_{ij} \mathbb{E}[\psi_i(X)\psi_j(X)]$ decomposes the dependence between the tests into two components: 
\begin{enumerate} 
\item the component $\mathbb{E}[\psi_i(X)\psi_j(X)]$ which measures the strength of association between the statistics $\psi_i$ and $\psi_j$ and therefore between $T_i$ and $T_j$, and 
\item the component $\varrho_{ij} = \lim\limits_{N\to\infty} \frac{\omega_{ij}(N)}{\sqrt{r_i(N)\, r_j(N)}}$ related to the extent of overlap between the two samples.
\item Moreover, the sign of $\Sigma_{ij}$ is either $0$ or the sign of $\mathbb{E}[\psi_i(X)\psi_j(X)]$. Thus, modifying the design of $\varrho_{ij}$ can never cause the sign to change from $\pm 1$ to $\mp 1$, but can force it to be $0$ through enforcing disjointness.
\end{enumerate}
\end{remark}
Thus, a \textit{sufficient} condition to reduce covariance is to reduce the term $\frac{\omega_{ij}}{\sqrt{r_i\, r_j}}$ via subsampling procedures. 

This theorem has several important corollaries for managing dependence. First, asymptotically, it is \textit{only} the fraction of the total sample $r_i$ and the overlap rates $\omega_{ij}$ and the influence functions $\psi_i$ that determine the dependence structure across tests.  Higher-order overlaps $|D_{i_1} \cap \cdots \cap D_{i_r}|$ for $r \geq 3$ do not appear since multivariate normal distribution is determined by its second moments. In Appendix \ref{sec:cumulants} we remark how the higher cumulants of $X$ appear for finite samples.

Moreover,
\begin{corollary} \label{cor:pearson_cor}
The asymptotic Pearson correlation of asymptotically linear test statistics $T_i$ with associated influence functions $\psi_i$ is bounded above by 

\begin{align}
\rho(T_i,T_j) \leq \frac{\omega_{ij}}{\sqrt{r_ir_j}} \frac{\mathbb{E}[\psi_i\psi_j]}{\sqrt{\mathbb{E}[\psi_i^2]\mathbb{E}[\psi_i^2]}} \leq  \frac{\omega_{ij}}{\sqrt{r_ir_j}}.
\end{align}
\end{corollary}

This corollary entails that a sampling technique that constrains $\omega_{ij}/\sqrt{r_i r_j}\to 0$ is sufficient to asymptotically control the dependence across tests.

\section{Asymptotic Implications for Mean-Variance Analysis of Error Distributions}

The joint normality established in Theorem~\ref{thm:asymptotic_linearity} has direct consequences for the distribution of the total error count. Suppose that $C$ two-sided tests are performed on a dataset at fixed level $\alpha$. Thus, the rejection events $R_i = \{|T_i| > c_{\alpha/2}\}$ are asymptotically determined by a multivariate normal vector with known covariance by Theorem \ref{thm:asymptotic_linearity}.

\subsection{Variance in the Distribution of Type I Errors under Nontrivial Dependence}

Under the global null with each test at level $\alpha$, the count of errors is given by the sum of the indicator functions for the rejection regions $R_i$. Consequently, if each test is performed at exact level $\alpha$, the \textit{expected} number of errors is by linearity of expectation
\begin{align}
\mathbb{E}[E] & = \mathbb{E}\left[\sum\limits_{i=1}^C \mathbf{1}_{R_i}\right] \\
&= C\alpha.
\end{align}
regardless of the dependence structure between the $R_i$. Thus, there is no difference in the \textit{mean} of the error distribution across possible dependence structures on $T_i$.

Our asymptotic results allow us to express the variance of this quantity in a closed form. Assume throughout this section that the test statistics $T_i$ are normalized to the standard Gaussian $\mathcal{N}(0,1)$, i.e. that they are $z$-statistics.

In particular, the variance of the Type I error count $E = \sum\limits_{i=1}^C \mathbf{1}_{R_i}$ satisfies

\begin{align}
\VV(E) & =  \underbrace{\sum\VV(\mathbf{1}_{R_i})}_{C\alpha(1-\alpha)} + 2\sum_{i< j} \underbrace{\Cov(\mathbf{1}_{R_i},\mathbf{1}_{R_j})}_{= \mathbb{P}(R_i\cap R_j)- \alpha^2}\\
& = C\alpha(1-\alpha) + 2\sum\limits_{i< j}[\mathbb{P}(R_i\cap R_j)-\alpha^2] \\
& \leq C^2\alpha(1-\alpha).
\end{align}

Observe that since $\mathbb{P}(R_i \cap R_j)\geq \alpha$, variance is maximized when $R_i$ and $R_j$ differ by a set of probability $0$: $\mathbb{P}[R_i \Delta R_j] = 0$. In that case the variance is equal to $\mathbb{V}(E) = C^2\alpha(1-\alpha)$.

To make this more practical, we need to compute $\mathbb{P}(R_i\cap R_j)$. Under normality, we can do so.\footnote{Assume that every normal test statistic is rescaled to be $\sim \mathcal{N}(0,1)$ under the null.}

\begin{proposition}\label{prop:correlated_rejection}
Let $T_i, T_j$ be jointly normal with unit variances and correlation 
$\rho_{ij} \in (-1,1)$. Let $R_i = \{|T_i| > c_{\alpha/2}\}$ and 
$R_j = \{|T_j| > c_{\alpha/2}\}$ be two-sided rejection events at 
level $\alpha$. Then:

\begin{enumerate}[label=\alph*)]
    \item The probability of pairwise joint rejection is
    \begin{align}\label{eq:joint_rejection}
    \mathbb{P}(R_i \cap R_j) 
    & = 4 - 8\Phi(c_{\alpha/2}) 
      + 2\Phi_{\rho_{ij}}(c_{\alpha/2}, c_{\alpha/2}) 
      + 2\Phi_{-\rho_{ij}}(c_{\alpha/2}, c_{\alpha/2}) \\
      & = \alpha^2 
      + 2\left(\Phi_{\rho_{ij}}(c_{\alpha/2}, c_{\alpha/2}) 
              + \Phi_{-\rho_{ij}}(c_{\alpha/2}, c_{\alpha/2}) 
              - 2\Phi(c_{\alpha/2})^2\right).
    \end{align}
    where $\Phi_{\rho}(x,y)$ denotes the CDF of the standard 
    bivariate normal with correlation $\rho$.
    
    \item $\mathbb{P}(R_i \cap R_j)$ is strictly increasing in 
    $\rho_{ij}$ on $(0,1)$.
    
    \item If $\rho_{ij} > 0$,
    \begin{align}
    \mathrm{Cov}(\mathbf{1}_{R_i}, \mathbf{1}_{R_j})> 0.
    \end{align}
\end{enumerate}
\end{proposition}

The proof of this proposition appears in Appendix \ref{proof:proof_of_prop:correlated_rejection}

When the tests are independent, then the variance is $C\alpha(1-\alpha)$. When the covariance matrix has $\Sigma_{ij}\geq 0$ for all $i,j$ (as in the case of the reuse of control group in two-sided sampling), variance will increase as $\mathbb{P}(R_i\cap R_j)\geq \alpha^2$.

As a corollary, we can show that if pairwise correlations are bounded above by $0 \leq \rho_{ij}\leq \rho_0$ that the increase in variance is bounded.

\begin{corollary}\label{cor:bounded_correlation_variance}
Suppose that all $(T_i)$ is jointly normal with unit variances and covariance matrix $\Sigma_{ij} = \rho_{ij}$ such that $0 \leq \rho_{ij}\leq \rho_0$ for $i\neq j$. Then the variance in the Type I error count $\mathbb{V}(E)$ under the global null is at most 

\begin{align}\label{eq:variance_excess_bound}
    \VV(E) \leq \underbrace{C\alpha(1-\alpha)}_{\sum\mathbb{V}(\mathbf{1}_{R_i})} + \underbrace{C(C-1)\, R(\rho_0, c_{\alpha/2})}_{\geq 2\sum\limits_{i<j} \Cov(\mathbf{1}_{R_i},\mathbf{1}_{R_j})}
    \end{align}
    where 
    \begin{align}
    R(\rho, c_{\alpha/2}) 
    & = 2\left(\Phi_{\rho}(c_{\alpha/2}, c_{\alpha/2}) 
             + \Phi_{-\rho}(c_{\alpha/2}, c_{\alpha/2}) 
             - 2\Phi(c_{\alpha/2})^2\right) \\
    & = \mathbb{P}[R_i\cap R_j] - \alpha^2.
    \end{align}
\end{corollary}

Thus, reducing positive dependence among tests by reducing covariance reduces the \textit{risk}---as measured by variance---in the Type I error distribution. In the extreme case of independent tests, $\mathbb{V}(E) = C\alpha(1-\alpha)$ scales linearly in $C$, while dependent tests have variance that scales \textit{quadratically} in $C$.

Moreover, the function $R(\rho,c_{\alpha/2})$ is bounded above by $\alpha(1-\alpha)\rho^2$ and is therefore subquadratic\footnote{Since $\frac{dR(\rho,c_{\alpha/2})}{d\rho}  = 2[\phi_{\rho}(c_{\alpha/2},c_{\alpha/2})-\phi_{-\rho}(c_{\alpha/2},c_{\alpha/2})]$ vanishes, the leading term in the Taylor series is quadratic of order $\rho^2$ with coefficient $\alpha(1-\alpha)$.}, as seen in Figure \ref{fig:subquadratic}

\begin{sidewaysfigure}[b]  
    \centering    \includegraphics[width=1.0\textwidth]{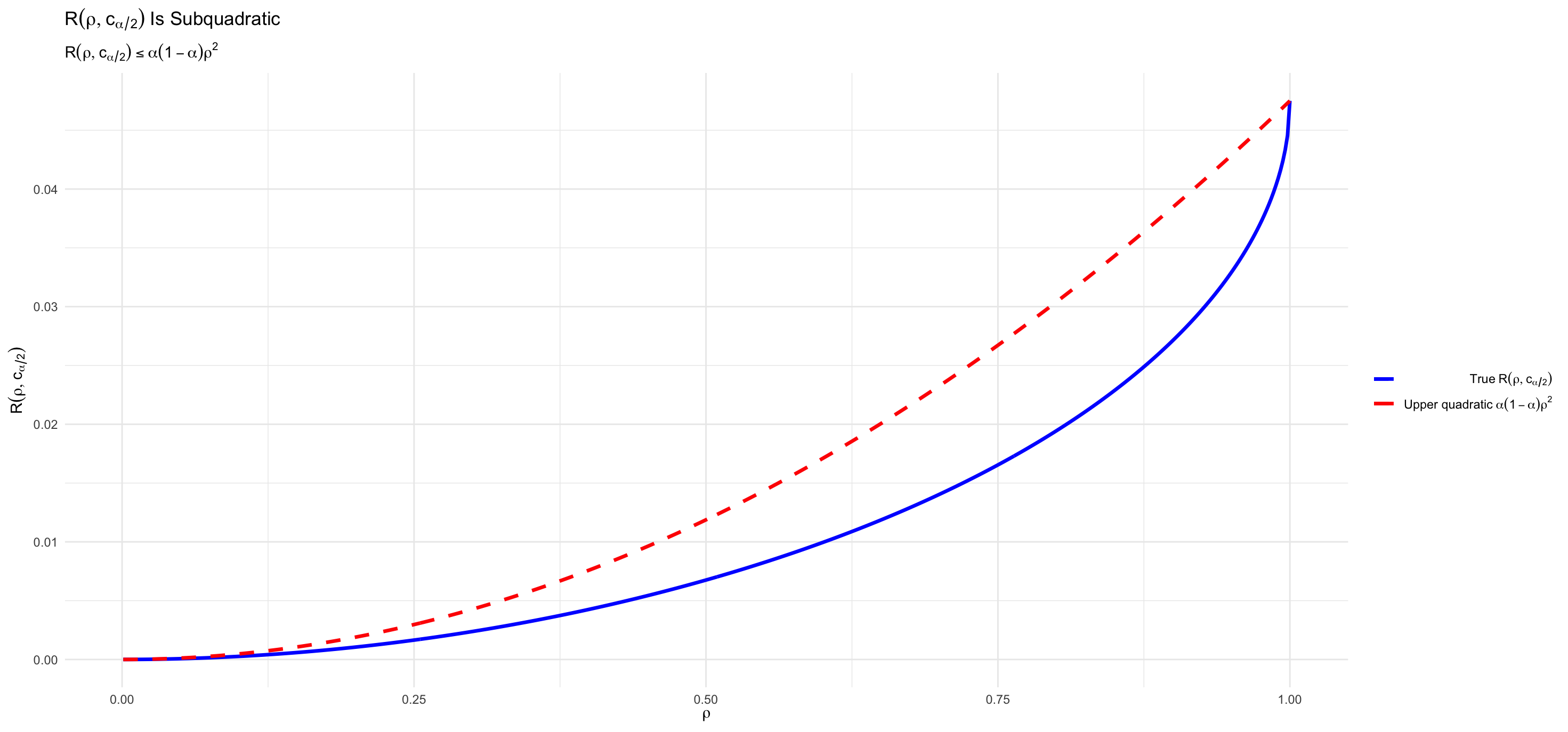}
    \caption{Subquadratic Growth of $R(\rho,c_{\alpha/2})$ in $\rho$}
    \label{fig:subquadratic}
\end{sidewaysfigure}

\subsection{Minimizing FWER Maximizes Expected Variance in the Error Distribution}

This result stands in contrast to that of familywise error rate, which is \textit{minimized} by increasing the variance of the Type I error distribution under the global null. Familywise error rate has already been criticized as being the wrong criteria by Efron \cite{efron_large-scale_2010} since it requires a low probability of a \textit{fixed} number of errors. Our objection stems from the fact that even for relatively small numbers of tests, familywise error control is optimized under the riskiest designs. This fact is implicitly acknowledged in the discussion of FWER in Fay \& Brittain Section 13.1 \cite{fay_statistical_2022}. More precisely, observe that

\begin{proposition}
Let $T_i$ be test statistics with rejection regions $R_i$. Then
\begin{align}
FWER & = \mathbb{P}\left( \bigcup R_i \right) \\
& \geq \alpha
\end{align}
with equality if and only if 
\begin{align}
\mathbb{P}(R_i \triangle R_j) =0
\end{align}
In particular, this occurs under the global null for tests $T_i \sim \mathcal{N}(0,1)$ if and only if $\rho_{ij} \equiv 1$.
\end{proposition}

Thus familywise error rate is optimized \textit{precisely} when the variance of the error distribution is highest. The following chart shows how familywise error rate and variance are negatively correlated for jointly normal tests $T_i \sim \mathcal{N}(0,1)$ with $\rho_{ij} \equiv \rho \geq 0.$

\begin{sidewaysfigure}[b]  
    \centering    \includegraphics[width=1.0\textwidth]{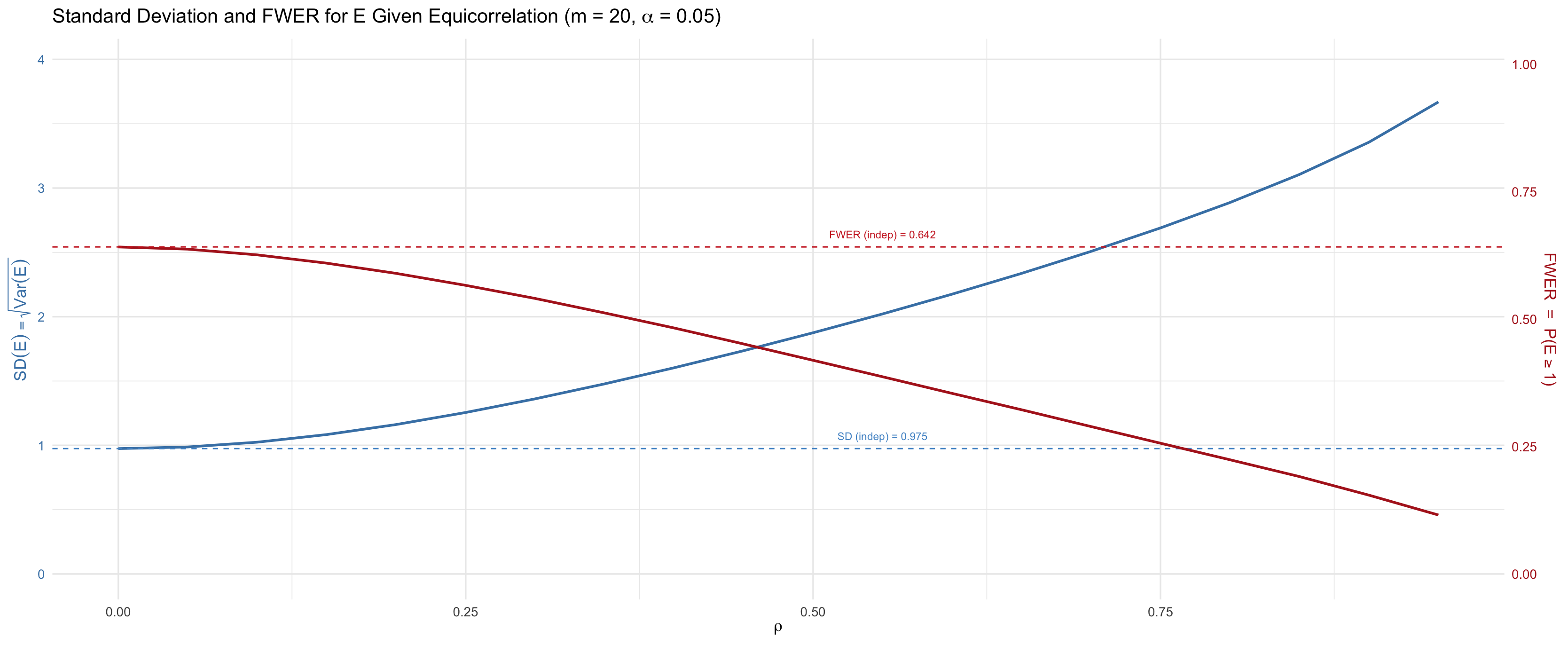}
    \caption{Standard Deviation vs. FWER for $E$ assuming $\Sigma_{ij} = \rho_{ij} = \rho$}
    \label{fig:fwer_vs_sd}
\end{sidewaysfigure}

The negative relationship observed above between FWER and variance of the Type I error distribution illustrated in Figure \ref{fig:fwer_vs_sd} is easy to understand: assume a fixed per-comparison error rate $\alpha$, the expected number of Type I errors is $C\alpha$ by linearity of expectation. However, a lower familywise error rate concentrates more probability mass on the event $E=0$, and therefore must spread out the remaining mass to larger values, increasing the variance.

\section{Subsampling and Decorrelation}\label{sec:subsampling}

In this section we analyze subsampling procedures that either exactly or asymptotically decorrelate a bounded (but possibly quite large) number $C$ of tests and consider the effects the procedures have on power. 

We stylize a bit by assuming that each test uses the \textit{same} inclusion and exclusion criteria, so that any observation $X_k$ can be an input of any test statistic $T_i$. This model works well for longitudinal cohort studies where the same patient population is followed and various outcomes are tracked over time, but does not describe the analytic situation of follow-on subgroup analyses (where only observations meeting more stringent inclusion criteria are accepted) or reused control groups (where novel therapies are compared only to a common control group). In those cases, the results of Theorem \ref{thm:asymptotic_linearity} still apply, and the analysis of this section applies \textit{mutatis mutandis} to those cases, as will be demonstrated in examples.

Corollary \ref{cor:pearson_cor} bounds the correlation between test statistics $T_i, T_j$ is bounded above by $\frac{\omega_{ij}}{\sqrt{r_i r_j}}$ where $\omega_{ij}$ is the size of overlap between $D_i$ and $D_j$ and $r_i = \frac{|D_i|}{|D|}$ is fixed. 

We consider two major classes of procedures: first, the method of \textit{data splitting} which guarantees correlation 0 across test statistics and second, \textit{independent uniform subsampling over $D$}. 

We will demonstrate that data splitting is in a formal sense asymptotically optimal as the size $N\to\infty$. However, there is a logistical catch: to implement data splitting requires a high degree of coordination across the investigators performing the $C$ studies to ensure that no data points overlap. Thus, data splitting is best applied in the context of a single investigator or a database with a centralized coordination process.

Next, we consider the procedures that assign to each test statistic $T_i$ a subsample that is drawn uniformly from $D$ of size $r(N)$. In these examples, the subsets may intersect and still have dependence among them; however, these procedures can be implemented by investigators independently of each other and are therefore amenable to deployment at the individual study level. We bound the variance contributed under such techniques in Corollary \ref{cor:variance_under_random_subsampling}.

\subsection{Data Splitting is Asymptotically Optimal, But Unlikely the Best Procedure in Practice}

In this section we discuss how the use of data splitting to partition a dataset $D$ into a fixed set $C$ of datasets of equal sizes is optimal by simultaneously decorrelating across test statistics while optimizing the minimum asymptotic relative efficiency across tests. 

We define the splitting procedure by:

\begin{definition}
   Let $N = |D|$ and assume $C \mid N$ and that $D = \{x_1,\dots, x_N\}$. The $C$-uniform data splitting procedure assigns to each test statistic $T_i$ the set 
   \begin{align}
   D_i = \{ x_j\in D|  j \equiv i \mod C\}.
   \end{align}
\end{definition}

First, ensuring disjoint data is necessary to ensure decorrelation across tests since common indices have a nontrivial expected contribution to the test $\propto\mathbb{E}[\psi_i(X_k) \psi_j(X_k)]$\footnote{That is, for a generic pair of asymptotically linear test statistics there is a distribution on $X$ for which $\mathbb{E}[\psi_i(X) \psi_j(X)] \neq 0$.} Thus, we can restrict to the family of subsampling procedures that produce a partition of $D$ into $C$ parts. Under mild regularity conditions, $C$-uniform splitting achieves the best possible asymptotic relative efficiency across tests.

\begin{proposition}\label{prop:asymptotic_optimality_of_splitting}
Let $S$ be a splitting procedure assigning $S(i) = D_i$. Let $T_i^s$ be the test statistic obtained by composing $T_i \circ S$, and let $T_i^f$ be the test statistic $T_i$ applied to $D$. 

The maximin asymptotic relative efficiency of a subsampling procedure partitioning $D$ into $C$ parts is 
\begin{align}
\max\min_{i} ARE(T^f_i,T_i^s) = \frac{1}{C}.
\end{align}

$C$-uniform data splitting obtains the maximum possible asymptotic relative efficiency.
\end{proposition}

The proof of this proposition is in Appendix \ref{proof:proof_of_prop:asymptotic_optimality_of_splitting}

Thus, $C$-uniform data splitting is optimal with respect maximin asymptotic relative efficiency across the tests while guaranteeing their independence.

Yet, while $C$-uniform data splitting has an asymptotic guarantee of optimality, it has two major disadvantages which do not recommend it as a practice in most real-world settings: first, enforcing non-overlapping splits of the dataset requires some sort of registry of the splits, to ensure that no reuse is occurring. Second, data splitting limits the number of analyses which can be performed to a relatively restricted number. In the next subsections we examine sampling with replacement as an alternative procedure which can be more realistically be implemented. In Section \ref{capacity} we develop the notion of the \textit{capacity} of a procedure as a way to describe the number of statistical tests which can be run.

\subsection{Uniform Independence Subsampling Techniques are Suboptimal, But More Logistically Feasible}

Despite the theoretical optimality of data splitting, we are going to consider policies that \textit{simultaneously} use less data (and are therefore less powerful) and are not guaranteed to render the test statistics independent. This may seem peculiar, but the reason is that data-splitting is ill-suited for being implemented across tests conducted by independent research teams: verifying that no data overlap has occurred is a challenge for publicly available datasets. In this section we define families of independent uniform subsampling techniques, and we analyze the combinatorics of these techniques in bounding the pairwise overlap rates that occur in the covariance decomposition of Theorem \ref{thm:asymptotic_linearity}. In this way we aim to quantify the suboptimality of the comparatively easy-to-implement subsampling techniques against the asymptotically optimal splitting techniques.

Therefore, we consider a class of suboptimal policies that can be implemented by a federated group of investigators that still guarantee asymptotic decorrelation based on uniform, independent subsampling from $D$ with fractions of data use given by $r_i(N), r_j(N) \in (0,1)$ depending on the size $N = |D|$ of the dataset. Under such subsampling procedures, for each $N$ we have $\mathbb{E}[\omega_{ij}(N)] = r_i(N) r_j(N)$ so that 
\begin{align}
\mathbb{E}\left[\left|\frac{\omega_{ij}(N)}{\sqrt{r_i(N) r_j(N)}}\right|\right] =  \sqrt{r_i(N) r_j(N)}
\end{align}
is the expected correlation. 

Observe that for any fixed $C > 0$ and set of functions $r_i(N): \mathbb{N} \to [0,1]$ converging to $0$ such that $Nr_i(N) \to \infty$ will yield asymptotic pairwise decorrelation and asymptotic power $1$ in by Theorem \ref{thm:asymptotic_linearity}. 

We also analyze two other methods: fixed \textit{rate} methods such as $r_i(N) \equiv q_i \in (0,1)$ which does \textit{not} guarantee decorrelation (their correlation is asymptotically $\sqrt{q_iq_j} > 0$) and fixed \textit{sample} techniques which guarantee rapid decorrelation but do not gain power in $N$. In particular, this means that independent uniform subsampling procedures with $Nr_i \in O(N)\setminus o(N)$ will never guarantee the independence of test statistics for finite $N$.

\begin{definition}\label{def:subsampling_families}
We define a subsampling procedure $S$ sampling $C$ subsets of $D$ with $|D|=N$ to be egalitarian provided 
\begin{enumerate}[label=\roman*]
    \item $\frac{|S_i(D)|}{N} = r(N)$ independent of $i$ ($r(N)$ is called the \textit{fraction function} of $S$),
    \item Each $S_i$ is subsampled uniformly from $D$,
    \item The $S_i$ are independent.
\end{enumerate}
\end{definition}

We have shown that $C$-uniform data splitting guarantees the independence of testing procedures and is optimal among such procedures with respect to maximin asymptotic relative efficiency in Prop \ref{prop:asymptotic_optimality_of_splitting}. By contrast, egalitarian subsampling procedures are suboptimal:

\begin{proposition}\label{prop:subsampling_suboptimal}

Let $S$ be an egalitarian subsampling procedure with fraction function $r(N)$ with $r(N) >0$ for all $N > C$. Then 

\label{item:subsampling_are} \begin{align}
\max\min_{i} ARE(T^f_i,T_i^s) = \lim\limits_{N\to\infty} r(N).
\end{align}

 In particular, if $\lim\limits r(N) \to 0$ then 
\begin{align}
\max\min_{i} ARE(T^f_i,T_i^s) = 0.
\end{align}
\end{proposition}

The proof follows \textit{mutatis mutandis} from the proof of Proposition \ref{prop:asymptotic_optimality_of_splitting}.

Egalitarian subsampling procedures are therefore asymptotically suboptimal in a strong sense.

\subsection{Bounded Suboptimality of Uniform Independent Subsampling Procedures for Finite Samples}

We now probe the question of the conditions under which egalitarian subsampling procedures are \textit{performant} insofar as their expected variance is within some tolerance of the corresponding independent portfolio. We formalize this as follows:

\begin{definition}
Let $S$ be a subsampling procedure, $(T_i)_{i\leq C}$ a set of test statistics, and $E = \sum \mathbf{1}_{R_i}$ the count of errors. Then we define the expected variance ratio (EVR) of $S$ to be
\begin{align}
   EVR(S) =  \frac{\mathbb{E}_S[\mathbb{V}(E)]}{C\alpha(1-\alpha)}.
\end{align}
\end{definition}

Observe that perfectly correlated rejection regions and full overlap yield an EVR equal to $C$ while an independent portfolio has EVR equal to $1$.

In this section we analyze the EVR for egalitarian subsampling procedures. Egalitarian subsampling procedures have random intersection sizes, so that the \textit{realized} variance of the portfolio depends on the draw from $S$, so bounds on EVR require an extra term appearing in the quadratic $C(C-1)$. This term depends on the pairwise probability of large intersections between two test statistics. We assume for the remainder of this section that the tests $T_i$ are all distributed via $\mathcal{N}(0,1)$ (a reasonable asymptotic assumption by Theorem \ref{thm:asymptotic_linearity}).

\begin{proposition}\label{cor:variance_under_random_subsampling}
Let $E = \sum_{i=1}^C \mathbf{1}_{R_i}$ be the Type I error count 
under the global null and assume that each $T_i \sim \mathcal{N}(0,1)$ under the null. Suppose that under the sampling procedure $S$ that for all $i \neq j$,
\begin{align}
\mathbb{P}(|\rho_{ij}| \geq \rho_0) \leq w = w(\rho_0,N).
\end{align}
Then 
\begin{align}
\mathbb{E}_S[\VV(E)] & \leq C\underbrace{\alpha(1-\alpha)}_{\mathbb{V}(\mathbf{1}_{R_i})} + C(C-1)\left[\alpha(1-\alpha)w(\rho_0)+ R(\rho_0, c_{\alpha/2})\right] 
\end{align}
where $R(\rho_0, c_{\alpha/2}) = 
2[\Phi_{\rho_0}(c_{\alpha/2},c_{\alpha/2}) + 
\Phi_{-\rho_0}(c_{\alpha/2},c_{\alpha/2}) - 2\Phi(c_{\alpha/2})^2]$.

In particular, if $w = \frac{\gamma}{\binom{C}{2}}$ then 
\begin{align}
\mathbb{E}_S[\mathbb{V}(E)] \leq (C+2\gamma)\alpha(1-\alpha) + C(C-1)R(\rho_0,c_{\alpha/2})
\end{align}
\end{proposition}

The proof appears in Appendix \ref{proof:proof_of_cor_variance_under_random_subsampling} This result enables a grid-search approach to optimizing the variance with levels for $w(\rho_0,N)$ and $\rho_0$ as a function of $\rho_0$ for fixed $N$ and will be used in the worked example in Section \ref{sec:worked_example}.

This immediately yields an upper bound on the EVR for a subsampling procedure $S$

\begin{corollary}\label{cor:evr_bound}
Let $E = \sum_{i=1}^C \mathbf{1}_{R_i}$ be the Type I error count 
under the global null and assume that each $T_i \sim \mathcal{N}(0,1)$ under the null. Suppose that under the sampling procedure $S$ that for all $i \neq j$,
\begin{align}
\mathbb{P}(|\rho_{ij}| \geq \rho_0) \leq w = w(\rho_0,N).
\end{align}
Then 
\begin{align}
EVR(S) \leq 1+\frac{(C-1)}{\alpha(1-\alpha)}\left[\alpha(1-\alpha)w(\rho_0,N)+ R(\rho_0, c_{\alpha/2})\right]
\end{align}
\end{corollary}

The remainder of this section will be spent investigating the contribution of the term $\mathbb{P}(|\rho_{ij}| \geq \rho_0)$.

\begin{definition}\label{def:bounded_suboptimality}
Let $C>0$ and let $(T_i) = T_1,\dots, T_C$ be a set of test statistics with fixed alternatives $(\theta_{1i}) \in \prod \Theta_i$. Let $D$ be a dataset of size $N$. A subsampling procedure $S$ is $(\rho_0,\beta_0,\gamma)$-performant for $T$ provided 

\begin{enumerate}[label=\roman*]
    \item\label{criterion:low_intersection} 
    (Low Probability of Large Pairwise Correlation) Under the subsampling procedure $S$, 
    \begin{align}
    \mathbb{P}(|\rho_{ij}| \geq \rho_0) \leq \frac{\gamma}{\binom{C}{2}}
    \end{align}
   
    \item\label{criterion:adequate_power} (Adequate Power Across Statistics) For every $T_i$, 
    \begin{align}
    \pi^S_i(\theta_i) \geq 1-\beta_0.
    \end{align}
    where $\pi^S_i(\theta_{1i})$ is the power of test $T_i$ when performing the test on $D_i = S_i(D)$.
\end{enumerate}

Define $M((T_i),(\theta_i),\beta_0)$ to be the maximum number of observations from $D$ required across $(T_i)$ under the fixed alternatives $(\theta_i)$ to be powered to $1-\beta_0$:
\begin{align}
M((T_i),(\theta_i),\beta_0) = \max\limits_{i\in C} \min\limits_{n \in \mathbb{N}} \{M \,|\, \pi_i(\theta_i,M) \geq 1- \beta_0  \}
\end{align}
\end{definition}

\begin{remark}
The relaxed criteria above reflect relevant considerations in applied analysis: analyses are conducted assuming power is sufficiently high \textit{and} we may tolerate some bounded level of correlation to avoid the logistical problems imposed by data splitting.

In particular, a $(\rho_0, \beta_0, \gamma)$-performant procedure incurs an increase in the expected number of Type II errors at most $C\beta_0$.
and ensures that the probability of \textit{at least one pair} of test statistics expected number of pairs with correlation exceeding $\rho_0$ is of order at most $\gamma$ by the union bound.
\end{remark}

\begin{remark}
In practice, when using a sufficiently large dataset a very low $\beta_0$ should be selected. If, for example, $\beta_0 = 0.2$, then $80\%$ power would be deemed performant. This incurs a substantial increase in expected Type II errors versus using the entire dataset, where the rate of Type II errors concentrates very close to $\beta \approx 0$. We suggest using maximally $\beta \leq 0.01$ since this incurs an expected cost of one Type II error per 100 tests conducted on the dataset versus the alternative of reusing all of $D$. In many cases, an even smaller $\beta$ will be appropriate.
\end{remark}




Immediate from the definition of $(\rho_0,\beta_0,\gamma)$-performant policies is that $C$-uniform data splitting is strongly performant:

\begin{proposition}\label{prop:data_splitting_is_performant}
$C$-uniform data splitting is $(0,\beta_0,0)$-performant provided 
\begin{align}
   |D| \geq CM((T_i),(\theta_i),\beta_0). 
\end{align}
\end{proposition}

The remainder of this section will be spent evaluating the conditions under which the subsampling techniques in Definition \ref{def:subsampling_families} are performant.

First, we relate the probability statement in Criterion \ref{criterion:low_intersection} to the size of the pairwise overlaps $|D_i\cap D_j|$

\begin{proposition}\label{prop:union_bound_performant}
Let $S$ be an egalitarian subsampling procedure on $D$ with $|D| = N$. Then

\begin{enumerate}
    \item $\mathbb{E}[|\rho_{{ij}}|] \leq  r(N).$
    \item The probability that $\rho_{ij}\geq \rho_0$ is bounded above by \begin{align} 
\mathbb{P}_S\left(|\rho_{ij}| \geq \rho_0\right) & \leq \mathbb{P}_S\left(\frac{\omega_{ij}}{\sqrt{r_ir_j}}\mathbb{E}[\psi_i\psi_j] \geq \rho_0\right) \\
& \leq \mathbb{P}_S\left(|D_i\cap D_j| \geq \frac{Nr(N)\rho_0}{\mathbb{E}[\psi_i\psi_j]}\right)
\end{align} 
\item $|D_i\cap D_j| \sim \text{Hypergeometric}(N,Nr(N), Nr(N)).$
\end{enumerate}


\end{proposition}

The proof is in Appendix \ref{proof:prop_union_bound_performant}.

It therefore remains to bound probabilities of the form $\mathbb{P}(|D_1\cap D_2| \geq Nr(N)\rho_0)$. We do so by exploiting tail inequalities of the Hypergeometric distribution whenever possible.

\begin{proposition}\label{prop:subsampling_tail_bound}
Let $S$ be an egalitarian subsampling procedure with associated subsample fraction $r(N) \in (0,1)$. 

Then 
\begin{enumerate}[label=\roman*]
    \item\label{item:expectation} $\mathbb{E}[|D_i\cap D_j|] = N r(N)^2$,
    \item\label{item:chernoff_prep}  (Preparation for Chernoff Bound) 
    \begin{align}
        \mathbb{P}\left(|D_1\cap D_2| \geq Nr(N)\rho_0\right) = \mathbb{P}\left(|D_1\cap D_2| \geq \frac{\rho_0}{r(N)} \mathbb{E}[|D_i\cap D_j|]\right).
    \end{align}

    \item\label{item:chernoff_bound} (Chernoff Bound). Let $\delta = \delta(\rho_0,r(N)) =  \frac{\rho_0}{r(N)}-1$. Then 
    \begin{align}
     \mathbb{P}\left(|D_1\cap D_2| \geq Nr(N)\rho_0\right) \leq \left(\frac{e^{\delta}}{(1+\delta)^{1+\delta}}\right)^{Nr(N)^2}
    \end{align}

    \item\label{item:hoeffding} (Hoeffding Bound) Let $\delta = \frac{\rho_0}{r(N)}-1$. Then 
    \begin{align}
     \mathbb{P}\left(|D_1\cap D_2| \geq Nr(N)\rho_0\right) \leq e^{-2\delta^2 Nr(N)^3}
    \end{align}

    \item\label{item:markov} (Markov Bound) 
    \begin{align}
     \mathbb{P}\left(|D_1\cap D_2| \geq Nr(N)\rho_0\right) \leq \frac{r(N)}{\rho_0}.
    \end{align}
\end{enumerate}
\end{proposition}

\begin{remark}
The Chernoff and Hoeffding bounds above are useful in different regimes. The Chernoff bound is most helpful when $r(N) \geq \frac{1}{\sqrt{N}}$ and $\rho_0 \geq r(N)$, in which case the expected pairwise intersection size is $Nr(N)^2 \geq 1$, making the exponent occurring in clause \ref{item:chernoff_bound} positive.

The bound obtained in clause $\label{item:hoeffding}$ is useful in the regime where $\delta^2Nr(N)^3 \geq 1$, which occurs when $\rho_0 \geq \frac{1}{\sqrt{N r(N)}}$. This yields nontrivial tail bounds even when $r(N) \leq \frac{1}{\sqrt{N}}$, but at the cost of a factor of $\frac{1}{\sqrt{r(N)}}$ in $\rho_0$.

In the case that $r(N) \leq \frac{1}{\sqrt{N}}$, one has a low probability of overlap but, when overlaps \textit{do} occur, they can have large relative sizes. In this case, the Markov bound outperforms either bound.\footnote{For a concrete simple example, if $N = 10^6$ ,and $Nr(N) = 5$, there is probability of order $10^{-5}$ that the two subsets intersect. However, conditional on intersecting nontrivially their relative intersection is of size at least $\frac{1}{5} = 0.2$.} 
\end{remark}

\begin{remark}
The above propositions are weaker than optimal if you know bounds on $|\mathbb{E}[\psi_i\psi_j]|$, in which case if $|\mathbb{E}[\psi_i\psi_j]| \leq \epsilon$ $\delta = \frac{\rho}{\epsilon r(N)}$. Our result is the degenerate case of $|\mathbb{E}[\psi_i\psi_j]| \leq 1$.
\end{remark}

The proof of this theorem appears in Appendix \ref{proof:proof_of_prop_subsampling_tail_bound}.

We let 
\begin{align} p_{Chernoff}(N,r(N),\rho_0) =  \left(\frac{e^{\delta(\rho_0,r(N))}}{(1+\delta(\rho_0,r(N)))^{1+\delta(\rho_0,r(N))}}\right)^{Nr(N)^2}
\end{align}

and 
\begin{align}
    p_{Hoeffding}(N,r(N),\rho_0) \leq e^{-2\delta(\rho_0,r(N))^2 Nr(N)^3}.
\end{align}

With 
\begin{align}
    p_{mixed}(N,r(N),\rho_0) = \min\{p_{Chernoff}(N,r(N),\rho_0), p_{Hoeffding}(N,r(N),\rho_0)\}
\end{align}
we have that 
\begin{align}
\mathbb{P}\left(|D_1\cap D_2| \geq Nr(N)\rho_0\right) \leq p_{mixed}(N,r(N),\rho_0).
\end{align}

These results allow us to bound the expected variance in proposition \ref{cor:variance_under_random_subsampling}  

\begin{corollary}\label{cor:bounding_variance}
Let $E = \sum_{i=1}^C \mathbf{1}_{R_i}$ be the Type I error count 
under the global null. Let $S$ be an egalitarian subsampling procedure with fraction function $r(N)$ and let $\rho_0 > r(N)$.
Then 
\begin{align}
\mathbb{E}_S[\VV(E)] & \leq  C\alpha(1-\alpha) + C(C-1)\left[\alpha(1-\alpha)p_{mixed}(N,r(N),\rho_0)+ R(\rho_0, c_{\alpha/2})\right]\\
\end{align}

and 
\begin{align}
EVR(S) & \leq 1+\frac{(C-1)}{\alpha(1-\alpha)}\left[\alpha(1-\alpha)p_{mixed}(N,r(N),\rho_0)+ R(\rho_0, c_{\alpha/2})\right].
\end{align}
\end{corollary}

Since $p_{mixed}(N,r(N),\rho_0)$ decays exponentially in $\frac{\rho_0}{r(N)}$ while $R(\rho_0,c_{\alpha/2})$ is approximately quadratic near $r(N)$, so heuristically

\begin{align}
EVR(S) \approx 1 +(C-1)R(r(N), c_{\alpha/2}).
\end{align} 

\subsection{Tolerating a Bounded Increase in Expected Variance Ratio Allows More Capacity to Perform Analyses} \label{capacity}

The results of the previous section established sufficient conditions for when egalitarian subsampling procedures are performant and therefore bound the $EVR$ of the subsampling procedure. In this section we examine how many studies $C$ a given subsampling procedure $S$ and dataset $D$ of size $N$ each requiring $\geq M$ observations can sustain while satisfying $EVR(S) \leq 1+{\delta}$. 

\begin{definition}\label{def:capacity}
Let $\delta >0$ and $S$ be a data allocation procedure defined for an arbitrary number of studies. We define the \textit{capacity} of $S$ together with dataset $D$ of size $N$ with per-study sample requirement $M$ to be

\begin{align}
C(S,N,M,\delta) = \max\{C \,\mid\, EVR(S,N) \leq 1+\delta \text{ and } |S_i(D)| \geq M.\}
\end{align}
\end{definition}

It is in this context where data splitting incurs a tradeoff: by enforcing \textit{exact} independence, data splitting allows only linear growth in the number of tests it can accommodate only  More precisely: $C$-uniform data splitting allows only $C \leq \lfloor \frac{N}{M} \rfloor$ by construction.\footnote{In analogy with egalitarian subsampling, we can think of data splitting as having a ``fraction function'' for the purposes of capacity: $r(N) = \frac{1}{M}$ where $M = M((T_i),(\theta_i),\beta_0)$ is the largest sample size required over a very large number of test statistics.}  
Allocating data linearly in $N$ results in a relatively few number of studies that are admissible. 

Egalitarian subsampling techniques are able to admit a far larger number of pairwise boundedly correlation test statistics due to the strong concentration bounds established in \ref{prop:subsampling_tail_bound}. 

\begin{proposition}\label{prop:egalitarian_evr_inversion}
Let $\delta \geq 0$, $S$ be an egalitarian subsampling procedure with fraction function $r(N) = \frac{b}{\sqrt{N}}$ and $N > \frac{M}{r(N)}$. Let $\rho_0 \geq r(N)$. Then for any number $C = C(S,N)$ of studies such that
\begin{align}
C(S,N) & \leq \frac{(\delta-1)\alpha(1-\alpha)}{\alpha(1-\alpha)p_{mixed}(N,r(N),\rho_0) + R(\rho_0,c_{\alpha/2})} \\
& \leq \frac{\delta-1}{\rho_0^2 + O(e^{-\Theta(\sqrt{N})})} \\
&\approx \frac{\delta-1}{(r(N)+\epsilon)^2} \\
& \approx \frac{(\delta-1)N}{b^2}
\end{align}
we have 
$EVR(S) \leq 1+\delta$.
\end{proposition}

The proof of this proposition appears in Appendix \ref{proof:proof_of_prop_egalitarian_evr_inversion}.

This provides a simple way to measure the capacity of the dataset dependent on (a) achievable $\rho_0$ (e.g. $\rho_0 = r(N)$) and (b) the tolerance $\delta$ for increased $EVR(S)$ by way of numerical optimizaton to find the optimal balance of $R(\rho_0,c_{\alpha/2})$ and $p_{mixed}(N,r(N),\rho_0)$.

\section{Worked Examples}\label{sec:worked_example}

\subsection{Revisiting the Common Control Group Problem}\label{subsec:reused_control}

In this section, we revisit the common control group problem to illustrate the tradeoffs and considerations in applying egalitarian subsampling techniques. We adopt the balanced design to compare the mean in treatment group $t_i$ to control $c$ via
\begin{align}
T_i = \frac{1}{\sqrt{n}}\sum X_{t,i} - \frac{1}{\sqrt{n}}\sum X_{c,i} 
\end{align}
where the $X_{t,i}$ are treatment units and $X_{c,i}$ are control units. We assume that there are $10$ treatment groups, so we are evaluating $C=10$ contrasts using this dataset.

We evaluate this hypothesis under the global null that each observation is iid drawn from $\mathcal{N}(0,1)$ irrespective of exposure. To achieve pairwise power of $1-\beta_0 = 99\%$ at level $\alpha = 0.05$ for an effect size of Cohen's $d = 0.2 = \frac{\mu_1-\mu_2}{\sigma}$ requires at least $n_{\text{ctrl}} = n_{\text{treat(i)}} = 920$ observations.


Suppose that we have $N = 10000$ observations for the control and each treatment arm. We analyze the variance in the distribution of errors under the global null under different subsampling procedures:

\begin{enumerate}
    \item Data Gluttony: Each $T_i$ is evaluated on all $10000$ observations in $t_i$ and all $10000$ in $c$
    \item Data Splitting: Each $T_i$ is evaluated on all $10000$ observations in $t_i$ and observations in sequence $[1000i ,1000i+999]$ in $c$.
    \item Egalitarian Subsampling: Each $T_i$ is evaluated on $b\sqrt{N}$ observations in $t_i$ with $b \in \{10,15,20\}$. 
\end{enumerate}

Under these conditions, each test is adequately powered to at least $1-\beta_0 = 99\%$. By the results in \cite{dunnett1955multiple}, in the data gluttony case the pairwise correlation between test statistics is $\rho_{ij} = \frac{1}{2}$ when $i\neq j$.

For data splitting, $\rho_{ij} = 0$ when $i\neq j$.

By Corollary \ref{cor:pearson_cor}, in the egalitarian subsampling case we have 
\begin{align}
\mathbb{E}[|\rho_{ij}| ] & \leq \frac{r(N)}{2}\mathbb{E}[T_iT_j] \\
& =\frac{b}{2\sqrt{N}} \\
& = \frac{b}{200}
\end{align}

with the factor of $\frac{1}{2}$ appearing because each study evaluates disjoint treatment groups.

Under data gluttony, the expected variance is equal to 
\begin{align}
\mathbb{V}_{gluttony}(E) & = 10(10-1)R(\rho = 0.5,c = c_{1-\alpha/2}) + 10\alpha(1-\alpha) \\
& \approx 1.083
\end{align}

Under data splitting, the variance is equal to the minimum attainable
\begin{align}
\mathbb{V}_{splitting}(E) & =  10\alpha(1-\alpha) \\
& = 0.475.
\end{align}

For the egalitarian subsampling procedures with $b\in \{10,15,20\}$ we use grid search to identify the $\rho_0(b)$ optimizing the term 
\begin{align}
\epsilon(b) := \alpha(1-\alpha)p_{mixed}(N,r(N),\rho_0)+ R(\rho_0, c_{\alpha/2}).
\end{align}
Applying the inequality in Corollary \ref{cor:bounding_variance}, we obtain the upper bounds for expected variance tabulated in Table \ref{tab:egalitarian_variance_control_reused_control}. We observe that in each case, the upper bound on the expected variance $EVR(S_b) \leq 1.1$

\begin{sidewaystable}[b]
\centering
\begin{tabular}{lccccccc}
\hline
Design & $n_{\text{treat}}$ & $n_{\text{ctrl}}$ & Power & $\mathbb{E}[\rho]$ & $\mathbb{E}_S[\mathbb{V}(E)]$ upper bound & optimal $\rho_0$ & Max $C$ ($\delta = 0.1$) \\
\hline
Data Gluttony & 10000 & 10000 & $\approx 1$ & 0.50 & 1.083 & -- & 1 \\
Data Splitting & 10000 & 1000 & $> 0.999$ & 0 & 0.475000 & -- & 10 \\
Egalitarian ($b=10$) & 1000 & 1000 & 0.994 & 0.05 & 0.487999 & 0.0729 & 33 \\
Egalitarian ($b=15$) & 1500 & 1500 & $> 0.999$ & 0.075 & 0.497848 & 0.0971 & 19 \\
Egalitarian ($b=20$) & 2000 & 2000 & $> 0.999$ & 0.10 & 0.510651 & 0.1215 & 12 \\
\hline
\end{tabular}
\caption{Egalitarian Designs with Variance Bounds and Capacity. Power is computed per study for Cohen's $d = 0.2$ at level $\alpha = 0.05$ (two-sided).}
\label{tab:egalitarian_variance_control_reused_control}
\end{sidewaystable}

From the perspective of capacity, suppose we adopt a tolerance for $EVR(S) = 1+\delta$ with $\delta = 0.1$ and maintaining power $1-\beta_0$ for each test. Then data gluttony is unfeasible since 
\begin{align}
\frac{\mathbb{V}_{gluttony}(E)}{\mathbb{V}_{splitting}(E)} & \approx \frac{1.083}{0.475} \\
& \approx 2.28 \\
& > 1+\delta.
\end{align}

For data splitting, $C=10$ is feasible because $\frac{N}{10} = 1000 > 920$ exceeds the minimum needed to achieve power in an equal-arms design. However, $\frac{N}{11} = 909 < 920$, so if future treatment groups had sample size $<920$ the resulting study would fail to achieve required power.

For the egalitarian procedures we apply  Corollary \ref{prop:egalitarian_evr_inversion} using the optimal $\rho_0(b)$ values obtained above to reach the values in Table \ref{tab:egalitarian_variance_control_reused_control}. Observe that the subsampling procedure $S_{10}$ can sustain over 30 pairwise contrasts while achieving EVR bounded above by $1.1$ with the same sized control group, increasing by a factor of $>3$.

\subsection{Families of Univariate Regressions} \label{subsec:sur}

In this example, we consider the evaluation of $C=10$ univariate linear regressions given by $X_i = \beta_i Y_i +\beta_{0i}$ evaluated by a standard two-sided $t$-test at level $\alpha = 0.05$. We assume that $X_1,\dots, X_C$ and $Y_1,\dots,Y_C$ are all distinct and moreover that
\begin{enumerate}
    \item (Standard Gaussian Marginals) All $X_i, Y_i \sim \mathcal{N}(0,1)$,
    \item (Global Null) $X_1,\dots, X_n \perp Y_1,\dots, Y_n$,
    \item (Within-Block Dependence) $(X_1,\dots,X_C) \sim \mathcal{N}(0,\Sigma_X)$ and $(Y_1,\dots,Y_C) \sim \mathcal{N}(0,\Sigma_Y)$.
\end{enumerate}

We investigate how the expected variance of the error distribution of Type I errors depends on the structure of $\Sigma_X,\Sigma_Y$. To do so, we used the \texttt{gencor} \texttt{R} package to uniformly sample from the space of correlation matrices satisfying bounds on the magnitudes of the pairwise correlations.
\begin{enumerate}
    \item (Random Correlational Structure) $I =[-1,1]$: random correlation structure
    \item (Moderate to Highly Correlated) $I = \left[1,-\frac{1}{3}\right] \cup \left[\frac{1}{3},1\right]$
    \item (Highly Correlated) $I = \left[1,-\frac{2}{3}\right] \cup \left[\frac{2}{3},1\right]$
    \item (Moderately Correlated) $I = \left[-\frac{2}{3},-\frac{1}{3}\right] \cup \left[\frac{1}{3},\frac{2}{3}\right]$
    \item (Moderately Positively Correlated) $I = \left[\frac{1}{3},\frac{2}{3}\right]$
    \item (Highly Positively Correlated) $I = \left[\frac{2}{3},1\right]$
\end{enumerate}

Under each $I$, we sampled $\Sigma_X$ and $\Sigma_Y$ once from the uniform distribution on correlation matrices subject to the constraint that for $i\neq j$ $(\Sigma_X)_{ij}, (\Sigma_Y)_{ij} \in I$.  Then we performed a simulation of $B=5000$ draws of $N=10000$ observations from each $X \sim \mathcal{N}(0,\Sigma_X)$ and $Y \sim \mathcal{N}(0,\Sigma_X)$. In each simulation we computed the $t$-statistics estimating $X_i = \beta_{i}Y_{i} + \beta_{0i}$ and recorded the count of Type I errors under Data Gluttony (use of all $N = 10000$ observations, Data Splitting (allocating $N = 1000$ per estimate), and Egalitarian subsampling with $r(N) = \frac{b}{\sqrt{N}}$ with $b\in \{10,15,20\}$.

The sampled $\Sigma_X$ and $\Sigma_Y$ under each $I$ are reported in Appendix \ref{sec:simulated_correlation_matrices}. The plug-in variance of the error distribution under each procedure and with respect to each correlation matrix are recorded in Table \ref{tab:sim_variance_table_sur}. Observe that the variance of Data Gluttony is highly contingent on the values of $I$. By contrast, data splitting and egalitarian subsampling procedures maintain good control of variance even in the case of extremely positively associated variables.

\begin{sidewaystable}[b]
\centering
\small
\begin{tabular}{lcccccc}
\hline
 & \multicolumn{4}{c}{Mixed-Sign $I$} & \multicolumn{2}{c}{Positive $I$} \\
Design & $[-1,\, 1]$ & $\pm[.333,\, 1]$ & $\pm[.666,\, 1]$ & $\pm[.333,\, .666]$ & $[.333,\, .666]$ & $[.666,\, 1]$ \\
\hline
Data Gluttony             & 0.482 & 0.786 & 1.791 & 0.553 & 0.575 & 1.672 \\
Data Splitting            & 0.481 & 0.478 & 0.466 & 0.470 & 0.482 & 0.481 \\
Egalitarian ($b=10$) & 0.477 & 0.470 & 0.472 & 0.460 & 0.463 & 0.484 \\
Egalitarian ($b=15$) & 0.459 & 0.470 & 0.501 & 0.464 & 0.464 & 0.487 \\
Egalitarian ($b=20$) & 0.474 & 0.494 & 0.532 & 0.488 & 0.466 & 0.542 \\
\hline
\end{tabular}
\caption{Simulated $\mathbb{E}_S[\mathbb{V}(E)]$ under the global null for families of univariate regressions ($C=10$, $N=10000$, $\alpha = 0.05$). the notation $\pm[a,b]$ denotes $[-b,-a] \cup [a,b]$. The independence baseline is $C\alpha(1-\alpha) = 0.475$.}
\label{tab:sim_variance_table_sur}
\end{sidewaystable}

\section{Discussion}

The results of this manuscript demonstrate that subsampling is a technique that can simultaneously mitigate the risks of dependent Type I error while maintaining satisfactory power. Moreover, in principle subsampling can be performed by an \textit{individual} researcher without a high degree of coordination. The main implementation risk with subsampling is the malicious use of iterating over draws of the subsample to $p$-hack an investigator's results. This can be mitigated through requiring a publicly verifiable, timestamped, randomly drawn \textit{seed} generating the subsample that an investigator is required to report.

The regime in which data splitting is effective is for large $N$ beginning with $N = 10^4$ as our examples show. This is bolstered by the fact that to achieve power $\geq99\%$ to detect even small effects by Cohen's $d$ (and other related measures such as $h$) is typically on the order of $10^3$ \cite{cohen2013statistical}. For substantially larger datasets, the capacity only increases as the per-study sample size required remains fixed to detect meaningful effect sizes.

Subsampling can also play a role in improving the rigor of observational studies by enabling out of sample verification. This has been suggested in the case of splitting to be used in the context of observational studies to mitigate the effects of researcher degrees of freedom by Dahl et al. \cite{dahl_data_2008} and can be adapted to the subsampling case.

Subsampling may not be appropriate in all cases. For cases of rare diseases or uncommon exposures, there may simply not be enough data in the whole world to enable a large number of uncorrelated studies. In such cases, we may \textit{tolerate} the sequential reuse of data but must acknowledge that the variance in error rate is increased. In such cases, techniques to ensure hypotheses are as uncorrelated as possible such as Rosenbaum's approach via evidence factors \cite{rosenbaum2010evidence} or Walker's approach via theory-driven orthogonal predictions \cite{walker_orthogonal_2010} are more appropriate for managing dependence across evaluated hypotheses.

This work can be extended in many natural ways. Future work will adapt these results to controlling the variance of the distribution of false discoveries under various subsampling protocols. Throughout this manuscript, we limited our focus to the case of the global null to bound the Type I error risk as it is the appropriate setting for minimizing the worst case scenario for false discoveries. It allows for the analysis of prospective design techniques that provably bound this error regardless of the underlying mixture of true or false hypotheses. We account for Type II errors as a constraint (e.g. in the definition of performant subsampling techniques). To wit, observe that Theorem \ref{thm:asymptotic_linearity} holds locally around the null assuming sufficient regularity and the contribution of the overlap term $\varrho_{ij}$ is common to both. What \textit{can} change is the contribution of the dependence term $\mathbb{E}[\psi_i(X)\psi_j(X)]$ under various alternatives, but the weight of this change is mitigated by reducing the term $\varrho_{ij}$ to near $0$ through subsampling.\footnote{Heuristically, the variance of the count of false discoveries under a mixture of true and false nulls is often substantially lower than that of the corresponding count of Type I errors under the global null. Under a fixed probability $\varpi_0$ of true nulls, the variance of the false discoveries has $\binom{\varpi_0 C}{2} = \frac{\varpi_0^2C^2 - \varpi_0 C}{2} \approx \varpi_0^2\binom{C}{2}$ many pairwise covariance terms appearing, so assuming that $\mathbb{E}[\psi_i(X)\psi_j(X)]$ is nearly constant across alternatives then the variance of the error distribution under the mixture can be far lower than that of the global null. The variance \textit{inflation} factor for the distribution of false discoveries under this mixture can then be analyzed similarly to our existing approach.}

We approached the problem from mean-variance theory, but stop-loss theory provides a more robust utility-theoretic account of rational preference relations between portfolios \cite{kaas_modern_2008}. Our focus in this paper has been on the \textit{sufficiency} of subsampling techniques, and future work can surely find improvements in the \textit{efficiency} of these procedures in less restrictive contexts.

\section{Acknowledgements}

All simulations and analyses were conducted in \texttt{R} version 4.3.1 \cite{Rcore2025}.
Multivariate normal sampling and bivariate normal CDF calculations 
used the \texttt{mvtnorm} package \cite{mvtnorm_pkg, genz2009mvtnorm}.
Correlation matrices were generated using the \texttt{gencor} package \cite{gencor_pkg}. All figures were produced with 
\texttt{ggplot2} \cite{ggplot2}. Claude Opus 4.6 \cite{anthropic_claude_opus_4_6} was used to review the manuscript draft and optimize code for simulations.

\newpage

\bibliographystyle{plain}
\bibliography{refs}

\newpage 
\appendix

\section{Proofs of Technical Results}\label{sec:proofs}

\subsection{Proof of Theorem \ref{thm:asymptotic_linearity}}

\begin{proof}\label{proof:proof_of_thm:asymptotic_linearity} Proof of Theorem \ref{thm:asymptotic_linearity}

The idea of the proof is simple: show that the linear approximants of $S_i^{(N)}$ of the form $c\sum\limits_{j} \psi_i(X_j)$ have the desired convergence to $\mathcal{N}(0,\Sigma)$ and invoke Slutsky's theorem to conclude that since  $S_i^{(N)} - L_i^{(N)} = o(1)$, $S^{(N)} \to \mathcal{N}(0,\Sigma)$ as well.

Define the leading linear approximations
\begin{align}
L_i^{(N)} & = \frac{1}{\sqrt{n_i^{(N)}}} \sum_{j \in D_i^{(N)}} \psi_i(X_j) \\
& = \frac{1}{\sqrt{n_i^{(N)}}} \sum\limits_{j \leq N} \underbrace{\psi_i(X_j)\mathbf{1}_{j\in D_i^{(N)}}}_{\psi_{i}(X_j)^*}
\end{align}
so that $S_i^{(N)} = L_i^{(N)} + o_p(1)$ by definition asymptotic linearity. We first show that $L^{(N)} := (L_1^{(N)}, \ldots, L_C^{(N)})$ converges jointly in distribution to $\mathcal{N}_C(0, \Sigma)$. By construction the $\psi_{i}(X_j)^*$ are independent but not necessarily identically distributed.

Note that the pairwise covariances are given by

\begin{align}
\Cov(L_i^{(N)}, L_j^{(N)}) 
& = \frac{1}{\sqrt{n_i^{(N)} n_j^{(N)}}} \sum_{k\in D_{i}^{(N)}}\sum_{\ell \in D_j^{(N)}}\Cov(\psi_i(X_k),\psi_j(X_{\ell}))
\end{align}

Terms of the form $\Cov(\psi_i(X_k),\psi_j(X_\ell)) $ with $k\neq \ell$ vanish since the $X_i$ and hence the influence functions $\psi_j(X_i)$ are independent. Terms of the form $\Cov(\psi_i(X_k),\psi_j(X_k))$ are equal to 

\begin{align}
\EE[\psi_i(X_k)\psi_j(X_k)] - \underbrace{\EE[\psi_i(X_k)]}_{=0}\underbrace{\EE[\psi_j(X_k)]}_{=0} = \EE[\psi_i(X_k)\psi_j(X_k)]
\end{align} 
since influence functions have zero mean. Thus

\begin{align}
\Cov(L_i^{(N)}, L_j^{(N)}) = \frac{1}{\sqrt{n_i^{(N)} n_j^{(N)}}} \sum_{k=1}^N (\mathbf{1}_{k \in D_i^{(N)}} \cdot \mathbf{1}_{k \in D_j^{(N)}}) \cdot \EE[\psi_i(X_k)\psi_j(X_k)]
\end{align}
Since the $X_k$ are identically distributed:
\begin{align}
\Cov(L_i^{(N)}, L_j^{(N)})&  = \frac{|D_i^{(N)} \cap D_j^{(N)}|}{\sqrt{n_i^{(N)} n_j^{(N)}}} \cdot \mathbb{E}[\psi_i(X)\psi_j(X)] \\
& \rightarrow \varrho_{ij} \mathbb{E}[\psi_i(X)\psi_j(X)]
\end{align}
This converges to $\Sigma_{ij} =\varrho_{ij}\cdot \mathbb{E}[\psi_i(X)\psi_j(X)]$ as $N \to \infty$. In particular, $\Sigma_{ii} = \mathbb{E}[\psi_i(X)^2] = \sigma_{\psi_i}^2$.

By the Cramer-Wold theorem (\cite{lehmann1999elements} Theorem 5.1.8) it suffices to show that for all $\lambda \in \mathbb{R}^C$ that the sum 
\begin{align}
\lambda\cdot L^{(N)} = \sum \lambda_i L_i^{(N)}
\end{align}
converges. Rewriting, this is equivalent to the convergence of
\begin{align}
\lambda\cdot L^{(N)} & = \sum\limits_{i=1}^{C} \lambda_i L_i^{(N)} \\
& = \sum\limits_{j=1}^N \underbrace{\sum\limits_{i=1}^{C} \frac{\lambda_i}{\sqrt{n_i^{(N)}}} \psi^*_i(X_j)}_{Y_j^{(N)}(\lambda) :=} \\
& = \sum\limits_{j=1}^{N} Y_j^{(N)}(\lambda).
\end{align}

The terms $Y_j^{(N)}(\lambda)$ are the sum of a fixed $C$ of bounded terms converging to $0$ at a rate of $O\left(\frac{1}{\min_{i\leq C}\sqrt{Nr_i(N)}}\right)$ which converges to $0$, so the Lindeberg-Feller condition is satisfied and the Lindeberg-Feller Central Limit Theorem (See for example \cite{van2000asymptotic} Proposition 2.27) ensures convergence.
Hence by the Cramer-Wold theorem the vector statistic 
\begin{align}
\left(L_i^{(N)}\right) \rightarrow \mathcal{N}(0,\Sigma)
\end{align}
in distribution. Thus $(S_i^{(N)}) \rightarrow \mathcal{N}(0,\Sigma)$ by Slutsky's lemma.
\end{proof}

\subsection{Proof of Proposition \ref{prop:correlated_rejection}}

\begin{proof}\label{proof:proof_of_prop:correlated_rejection} (Proof of Proposition \ref{prop:correlated_rejection})

By construction, $R_i\cap R_j$ decomposes as the union of four disjoint events:

\begin{align}
R_i\cap R_j &= \underbrace{\{T_i > c_{\alpha/2}, T_j > c_{\alpha/2}\}}_{:= R_{i,j}(++)} \\ 
& \cup \underbrace{\{T_i < -c_{\alpha/2}, T_j > c_{\alpha/2}\}}_{:= R_{i,j}(-+)} \\
& \cup \underbrace{\{T_i > c_{\alpha/2}, T_j < -c_{\alpha/2}\}}_{:= R_{i,j}(+-)} \\
& \cup \underbrace{\{T_i < -c_{\alpha/2}, T_j < - c_{\alpha/2}\}}_{:= R_{i,j}(--)} 
\end{align}

Since the normal distribution is symmetric about its mean, 
\begin{align} \mathbb{P}(R_{i,j}(++)) = \mathbb{P}(R_{i,j}(--))\\  \mathbb{P}(R_{i,j}(+-)) = \mathbb{P}(R_{i,j}(-+))
\end{align}

Moreover, with $\Phi(x)$ being the cdf of the standard normal and $\Phi_{\rho}(x,y)$ the cdf of the bivariate normal with correlation $\rho$. Then with $\rho_{ij}$ the correlation between $T_i$ and $T_j$ we have (by inclusion exclusion)
\begin{align}
\mathbb{P}(R_{i,j}(++)) = 1-2\Phi(c_{\alpha/2}) + \Phi_{\rho_{ij}}(c_{\alpha/2},c_{\alpha/2}) \\
\mathbb{P}(R_{i,j}(+-)) = 1-2\Phi(c_{\alpha/2}) + \Phi_{-\rho_{ij}}(c_{\alpha/2},c_{\alpha/2}).
\end{align}

Thus
\begin{align}
\mathbb{P}[R_i\cap R_j] = 4 - 8\Phi(c_{\alpha/2}) 
      + 2\Phi_{\rho_{ij}}(c_{\alpha/2}, c_{\alpha/2}) 
      + 2\Phi_{-\rho_{ij}}(c_{\alpha/2}, c_{\alpha/2}).
\end{align}

Observing that $\alpha^2 = (2(1-\Phi(c_{\alpha/2})))^2$ we conclude that 
\begin{align}
\mathbb{P}[R_i\cap R_j] = \alpha^2 
      + 2\left(\Phi_{\rho_{ij}}(c_{\alpha/2}, c_{\alpha/2}) 
              + \Phi_{-\rho_{ij}}(c_{\alpha/2}, c_{\alpha/2}) 
              - 2\Phi(c_{\alpha/2})^2\right)
\end{align}

We now show that $\mathbb{P}(R_i \cap R_j)$ is strictly increasing in $\rho_{ij}$ on $(0,1)$. Let 
\begin{align}
Q(\rho,c_{\alpha/2}) & = \mathbb{P}[R_i\cap R_j] \\
& =  4 - 8\Phi(c_{\alpha/2}) 
      + 2\Phi_{\rho_{ij}}(c_{\alpha/2}, c_{\alpha/2}) 
      + 2\Phi_{-\rho_{ij}}(c_{\alpha/2}, c_{\alpha/2}).
\end{align}

Now we argue that $Q(\rho, c_{\alpha/2})$ is monotonically increasing for $\rho \in (0,1)$. Fixing $c_{\alpha/2}$, we have that

\begin{align}
\frac{dQ}{d\rho} &= 2\left(\frac{d\Phi_{\rho}(c_{\alpha/2},c_{\alpha/2})}{d\rho} - \frac{d\Phi_{-\rho}((c_{\alpha/2},c_{\alpha/2}))}{d\rho}\right) \\
&= \frac{2}{2\pi\sqrt{1-\rho^2}}\left(\underbrace{\exp\left(-\frac{c_{\alpha/2}^2 - 2\rho c_{\alpha/2}^2 + c_{\alpha/2}^2}{2(1-\rho^2)}\right)}_{\frac{d\Phi_{\rho}(c_{\alpha/2},c_{\alpha/2})}{d\rho} \text{ term}} - \underbrace{\exp\left(-\frac{c_{\alpha/2}^2 + 2\rho c_{\alpha/2}^2 + c_{\alpha/2}^2}{2(1-\rho^2)}\right)}_{\frac{d\Phi_{-\rho}(c_{\alpha/2},c_{\alpha/2})}{d\rho} \text{ term}} \right) \\
& = \frac{2}{2\pi\sqrt{1-\rho^2}}\left(\underbrace{\exp\left(-\frac{2c_{\alpha/2}^2(1-\rho)}{2(1-\rho^2)}\right)}_{\frac{d\Phi_{\rho}(c_{\alpha/2},c_{\alpha/2})}{d\rho} \text{ term}} - \underbrace{\exp\left(-\frac{2c_{\alpha/2}^2(1+\rho)}{2(1-\rho^2)}\right)}_{\frac{d\Phi_{-\rho}(c_{\alpha/2},c_{\alpha/2})}{d\rho} \text{ term}} \right) \\
& = \frac{2}{2\pi\sqrt{1-\rho^2}}\left(\underbrace{\exp\left(-\frac{c_{\alpha/2}^2}{1+\rho}\right)}_{\frac{d\Phi_{\rho}(c_{\alpha/2},c_{\alpha/2})}{d\rho} \text{ term}} - \underbrace{\exp\left(-\frac{c_{\alpha/2}^2}{1-\rho}\right)}_{\frac{d\Phi_{-\rho}(c_{\alpha/2},c_{\alpha/2})}{d\rho} \text{ term}} \right).
\end{align}

Since $\rho >0$ we have $\exp\left(-\frac{c_{\alpha/2}^2}{1+\rho}\right) > \exp\left(-\frac{c_{\alpha/2}^2}{1-\rho}\right)$ so that $\frac{dQ}{d\rho} > 0$ as desired.

Now, since $\mathbb{P}[R_i\cap R_j]$ is monotonically increasing in $\rho_{ij}$ for $\rho_{ij} > 0$, to show that $\Cov(\mathbf{1}_{R_i},\mathbf{1}_{R_j}) = \mathbb{P}[R_i\cap R_j] - \alpha^2 > 0$ if $\rho_{ij} > 0$ we need only show that $\mathbb{P}[R_i\cap R_j] = \alpha^2$ when they are independent. But this is true since each indicator function $\mathbf{1}_{R_k}$ has probability $\alpha$.
\end{proof}

\subsection{Proof of Proposition \ref{prop:asymptotic_optimality_of_splitting}}

\begin{proof} \label{proof:proof_of_prop:asymptotic_optimality_of_splitting} Proof of \ref{prop:asymptotic_optimality_of_splitting}
The asymptotic relative efficiency of $ARE(T^f_i,T_i^s)$ for any splitting procedure is given by the squared ratio of the standard deviations $\left(\frac{\sigma_{T^f_i}}{\sigma_{T^s_i}}\right)^2$ (by Theorem 14.7 in \cite{van2000asymptotic}). By construction, $\sigma_{T^s_i} = \frac{\sigma_{T^f_i}}{\sqrt{\frac{|D_i|}{N}}}$ so $ARE(T^f_i,T^s_i) = \frac{|D_i|}{N}$.

Since $\sum\limits_{i=1}^C |D_i| = N$, 
\begin{align}
\min \frac{D_i}{N} \leq \frac{1}{C}
\end{align}
and so 

\begin{align}
\max\min_{i} ARE(T^f_i,T_i^s) = \frac{1}{C}.
\end{align}

By construction $C$-uniform data splitting obtains this value of ARE across all tests $T_i$.
\end{proof}

\subsection{Proof of Proposition \ref{cor:variance_under_random_subsampling}}

\begin{proof}\label{proof:proof_of_cor_variance_under_random_subsampling}
Proof of Proposition \ref{cor:variance_under_random_subsampling}
By linearity of expectation and the assumption that the tests $T_i$ are exactly level $\alpha$, 

\begin{align}
\mathbb{E}_S[E] = C\alpha
\end{align}
for any subsampling procedure. Thus, by the law of total variance 
\begin{align}
\VV(E) & = \EE_S[\VV(E|S)] + \underbrace{\mathbb{V}_S(\mathbb{E}(E))}_{ =0 \text{ since }\mathbb{E}(E) \equiv C\alpha} \\
&= C\alpha(1-\alpha) + 2\sum\limits_{i<j} \EE_S[R(\rho_{ij}, c_{\alpha/2})].
\end{align}

For each pair $i\neq j$ we condition on the relation between $|\rho_{ij}|$ and $\rho_0$:
\begin{align}
\EE_S[R(\rho_{ij},c_{\alpha/2})] 
&= \EE[R(\rho_{ij},c_{\alpha/2})\mathbf{1}_{|\rho_{ij}|<\rho_0}] 
 + \EE[R(\rho_{ij},c_{\alpha/2})\mathbf{1}_{|\rho_{ij}|\geq\rho_0}].
\end{align}
By the monotonicity of $R$ (Proposition 
\ref{prop:correlated_rejection}) we have
$R(\rho_{ij},c_{\alpha/2}) \leq R(\rho_0,c_{\alpha/2})$. In the worst case $|\rho_{ij}|\geq \rho_0$, monotonicity yields
$R(\rho_{ij},c_{\alpha/2}) \leq R(1,c_{\alpha/2}) = \alpha(1-\alpha)$ since when 
$\rho_{ij} = 1$ we have 
$\mathbb{P}(R_i \cap R_j) = \alpha$. Thus
\begin{align}
\EE[R(\rho_{ij})] 
\leq R(\rho_0,c_{\alpha/2}) + \alpha(1-\alpha)\, w.
\end{align}
Summing over each of the $\binom{C}{2}$ distinct pairs yields the desired sum.

Substituting $w = \frac{\gamma}{\binom{C}{2}}$ yields 
\begin{align}
\mathbb{E}_S[\mathbb{V}(E)] \leq (C+2\gamma)\alpha(1-\alpha) + C(C-1)R(\rho_0,c_{\alpha/2})
\end{align}
\end{proof}

\subsection{Proof of Proposition \ref{prop:union_bound_performant}}

\begin{proof}\label{proof:prop_union_bound_performant} This is the proof of proposition \ref{prop:union_bound_performant}.
Let $S$ be egalitarian. 
For any finite sample $N$, the pairwise correlation between $T_i$ and $T_j$ is bounded above by 
\begin{align}
|\rho_{ij}| & \leq \frac{\omega_{ij}(N)}{\sqrt{r(N)^2}} \\ 
& =\frac{\omega_{ij}(N)}{r(N)} \\
& = \frac{N\omega_{ij}}{Nr(N)} \\
& = \frac{|D_i\cap D_j|}{Nr(N)}.
\end{align}
so that 
\begin{align}
\sum\limits\limits_{i\neq j \leq C}\mathbb{P}\left( |\rho_{ij}| \geq \rho_0 \right) \leq \sum\limits\limits_{i\neq j \leq C} \mathbb{P}\left( |D_i\cap D_j| \geq Nr(N)\rho_0 \right)
\end{align}

By egalitarianism of $S$, in particular the subsets $D_i$ and $D_j$ are selected uniformly from the set of subsets of $D$ of size $Nr(N)$, for any pairs of $i\neq j$ 
\begin{align}
    |D_i\cap D_j| \sim \text{Hypergeometric}(N,Nr(N), Nr(N)).
\end{align}
as desired.
\end{proof}

\subsection{Proof of Proposition \ref{prop:subsampling_tail_bound}}

\begin{proof}\label{proof:proof_of_prop_subsampling_tail_bound}
First, item \ref{item:expectation} follows since the subsamples $D_i$ and $D_j$ are independent and so the probability 
\begin{align}
\mathbb{P}[x\in D_i\cap D_j] = \mathbb{P}[x\in D_i]\mathbb{P}[x\in D_j] = r(N)^2.
\end{align}
so 
\begin{align}
\mathbb{E}[|D_i\cap D_j|] = N r(N)^2
\end{align}

Item \ref{item:chernoff_prep} follows directly from item \ref{item:expectation}. 

Next, item \ref{item:chernoff_bound} is an immediate consequence of the variant of the Chernoff bound for hypergeometric random variables in \cite{mulzer2018chernoff} Theorem 5.3 and the proof of Corollary 4.2 from Theorem 2.1. \ref{item:expectation}.

Finally, clause \ref{item:hoeffding} follows from the Hoeffding bound \cite{hoeffding1963probability} v, which in our context says

\begin{align}
 \mathbb{P}\left(|D_1\cap D_2| \geq Nr(N)\rho_0\right)  & \leq e^{-2\frac{\delta^2\mathbb{E}[|D_1\cap D_2|]^2}{Nr(N)}} \\
 & = e^{-2\frac{\delta^2 N^2 r(N)^4}{Nr(N)}} \\
 & = e^{-2\delta^2 N r(N)^3}.
\end{align}

Finally, the Markov bound applies via

\begin{align}
\mathbb{P}\left(|D_1\cap D_2| \geq Nr(N)\rho_0\right) &\leq \frac{\mathbb{E}(|D_1\cap D_2|)}{Nr(N)\rho_0}\\
& = \frac{Nr(N)^2}{Nr(N)\rho_0} \\
& = \frac{r(N)}{\rho_0}.
\end{align}
\end{proof}

\subsection{Proof of Proposition \ref{prop:egalitarian_evr_inversion}}

\begin{proof}\label{proof:proof_of_prop_egalitarian_evr_inversion}
First, constraint $EVR(S) \leq 1+\delta$ 
and rearranging the terms appearing in Corollary \ref{cor:bounding_variance} to obtain

\begin{align}
C& \leq \frac{(\delta-1)\alpha(1-\alpha)}{\alpha(1-\alpha)p_{mixed}(N,r(N),\rho_0) + R(\rho_0,c_{\alpha/2})} 
\end{align}.

Now observe that since $p_{mixed}  = O(e^{\Theta(-\sqrt{N}})$ and $R(\rho_0,c_{\alpha/2}) \leq \alpha(1-\alpha)\rho_0^2$ that for sufficiently large $N$ with $\rho_0 \approx r(N)$

\begin{align}
\frac{(\delta-1)\alpha(1-\alpha)}{\alpha(1-\alpha)p_{mixed}(N,r(N),\rho_0) + R(\rho_0,c_{\alpha/2})} \approx \frac{\delta-1}{r(N)^2}.
\end{align}

Thus, any $C$ lower than this will ensure $\mathbb{E}_S[\mathbb{V}(E)] \leq 1+\delta$.
\end{proof}

\section{Contribution of Higher Cumulants for Linear Statistics}\label{sec:cumulants}

The asymptotic results above Theorem \ref{thm:asymptotic_linearity} should be couched in an analysis of how higher-order overlap affects exact linear statistics non-asymptotically. 

\textit{Non-asymptotically} the higher-order dependence of the linear test statistics $T_i$ depend on the higher-order cumulants in an easy way. Assume that the $X_i$ have mean $0$, variance $1$ are iid. Let $c_m := \kappa_m(X)$ be the $m^{\text{th}}$ cumulant of the distribution of $X$.  Let $T_i = \sum\limits_{k} a_{i,j} X_j$ be a linear test statistic. Then since cumulants are multilinear (see, e.g., \cite{brillinger1992moments}), 
\begin{align}
\kappa_j(T_{i_1},\dots, T_{i_j}) = c_j \sum\limits_{k=1}^{|D|} a_{i_1,k} \cdots a_{i_j,k} .
\end{align}

For the test statistic $T_i = \frac{\sqrt{|D_i|}}{\sigma} \left(\sum\limits_{j\in D_i} X_j - \mu \right)$ this reduces to 
\begin{align}
\kappa_j(T_{i_1},\dots, T_{i_j}) = c_j \frac{|D_{i_1} \cap \cdots \cap D_{i_j}|}{\sqrt{|D_{i_1}| \cdots |D_{i_j}|}}
\end{align}

Assuming that each $D_i$ is a fraction $r$ of the dataset $D$ and are drawn independently and uniformly, this reduces to 
\[ \frac{|D_{i_1} \cap \cdots \cap D_{i_j}|}{\sqrt{|D_{i_1}| \cdots |D_{i_j}|}} \approx \frac{|D|\times r^j}{\sqrt{|D|^j}\times \sqrt{r^j}} = N^{1-j/2} r^{j/2}.\]

Thus, the higher cumulants tend to vanish polynomially in the parameters $r$ and $N$ for linear statistics.

\section{Sampled Correlation Matrices for the Families of Univariate Regressions Example}\label{sec:simulated_correlation_matrices}

This appendix reports the realized correlation matrices $\Sigma_X$ (among $X_1,\ldots,X_{10}$) and $\Sigma_Y$ (among $Y_1,\ldots,Y_{10}$) used in the families of univariate regressions simulation. Matrices are generated using the \texttt{gencor} package, which constructs positive-definite correlation matrices by calibrating the standard deviations of underlying normal random variables. All matrices are PSD by construction. Seeds are fixed for reproducibility.
 
\medskip
\noindent\textbf{Notation.} $\pm[a,b]$ denotes correlations sampled from $[-b,-a]\cup[a,b]$ with random signs. ``$[a,b]$ pos'' denotes all-positive correlations in $[a,b]$. $\lambda_{\min}$ is the smallest eigenvalue; $\overline{|\rho|}$ is the mean absolute off-diagonal entry.
 
\subsection*{$I = [-1,1]$}
 
\begin{equation*}
\renewcommand{\arraystretch}{0.85}
\Sigma_X = \scriptsize\begin{pmatrix}
 1.00 &  0.04 & -0.10 & -0.02 & -0.03 &  0.01 &  0.02 &  0.02 & -0.04 &  0.04 \\
 0.04 &  1.00 & -0.11 & -0.09 & -0.04 & -0.04 &  0.00 & -0.01 & -0.04 & -0.03 \\
-0.10 & -0.11 &  1.00 &  0.39 &  0.17 &  0.18 &  0.08 & -0.08 &  0.20 &  0.16 \\
-0.02 & -0.09 &  0.39 &  1.00 &  0.07 &  0.04 &  0.06 & -0.03 &  0.15 &  0.02 \\
-0.03 & -0.04 &  0.17 &  0.07 &  1.00 &  0.05 &  0.02 & -0.01 &  0.07 &  0.08 \\
 0.01 & -0.04 &  0.18 &  0.04 &  0.05 &  1.00 &  0.01 & -0.01 &  0.00 & -0.02 \\
 0.02 &  0.00 &  0.08 &  0.06 &  0.02 &  0.01 &  1.00 & -0.04 &  0.02 &  0.08 \\
 0.02 & -0.01 & -0.08 & -0.03 & -0.01 & -0.01 & -0.04 &  1.00 &  0.02 &  0.02 \\
-0.04 & -0.04 &  0.20 &  0.15 &  0.07 &  0.00 &  0.02 &  0.02 &  1.00 &  0.07 \\
 0.04 & -0.03 &  0.16 &  0.02 &  0.08 & -0.02 &  0.08 &  0.02 &  0.07 &  1.00
\end{pmatrix}
\end{equation*}
 
\begin{equation*}
\renewcommand{\arraystretch}{0.85}
\Sigma_Y = \scriptsize\begin{pmatrix}
 1.00 &  0.08 &  0.04 & -0.06 &  0.01 &  0.11 & -0.03 &  0.01 & -0.02 &  0.00 \\
 0.08 &  1.00 &  0.17 & -0.09 &  0.07 &  0.26 & -0.05 &  0.00 &  0.03 &  0.06 \\
 0.04 &  0.17 &  1.00 & -0.04 &  0.06 &  0.39 & -0.06 & -0.08 &  0.09 &  0.10 \\
-0.06 & -0.09 & -0.04 &  1.00 & -0.02 & -0.12 & -0.01 & -0.03 & -0.06 &  0.02 \\
 0.01 &  0.07 &  0.06 & -0.02 &  1.00 &  0.07 &  0.00 & -0.02 &  0.00 & -0.02 \\
 0.11 &  0.26 &  0.39 & -0.12 &  0.07 &  1.00 & -0.11 & -0.09 &  0.08 &  0.13 \\
-0.03 & -0.05 & -0.06 & -0.01 &  0.00 & -0.11 &  1.00 &  0.01 & -0.05 & -0.06 \\
 0.01 &  0.00 & -0.08 & -0.03 & -0.02 & -0.09 &  0.01 &  1.00 & -0.02 &  0.01 \\
-0.02 &  0.03 &  0.09 & -0.06 &  0.00 &  0.08 & -0.05 & -0.02 &  1.00 & -0.05 \\
 0.00 &  0.06 &  0.10 &  0.02 & -0.02 &  0.13 & -0.06 &  0.01 & -0.05 &  1.00
\end{pmatrix}
\end{equation*}
 
\begin{center}\small
\begin{tabular}{lccc}
\hline
Matrix & $\overline{|\rho|}$ & $\lambda_{\min}$ \\
\hline
$\Sigma_X$ & 0.063 & 0.535 \\
$\Sigma_Y$ & 0.065 & 0.581 \\
\hline
\end{tabular}
\end{center}
 
\smallskip
 
\subsection*{$I = \pm[0.33,\,1]$}
 
\begin{equation*}
\renewcommand{\arraystretch}{0.85}
\Sigma_X = \scriptsize\begin{pmatrix}
 1.00 &  0.46 & -0.67 & -0.63 & -0.55 & -0.48 & -0.39 &  0.45 & -0.56 & -0.42 \\
 0.46 &  1.00 & -0.66 & -0.64 & -0.55 & -0.49 & -0.40 &  0.43 & -0.56 & -0.45 \\
-0.67 & -0.66 &  1.00 &  0.94 &  0.82 &  0.73 &  0.61 & -0.67 &  0.83 &  0.67 \\
-0.63 & -0.64 &  0.94 &  1.00 &  0.77 &  0.68 &  0.58 & -0.63 &  0.80 &  0.62 \\
-0.55 & -0.55 &  0.82 &  0.77 &  1.00 &  0.61 &  0.50 & -0.54 &  0.69 &  0.57 \\
-0.48 & -0.49 &  0.73 &  0.68 &  0.61 &  1.00 &  0.45 & -0.49 &  0.60 &  0.47 \\
-0.39 & -0.40 &  0.61 &  0.58 &  0.50 &  0.45 &  1.00 & -0.43 &  0.51 &  0.45 \\
 0.45 &  0.43 & -0.67 & -0.63 & -0.54 & -0.49 & -0.43 &  1.00 & -0.54 & -0.43 \\
-0.56 & -0.56 &  0.83 &  0.80 &  0.69 &  0.60 &  0.51 & -0.54 &  1.00 &  0.57 \\
-0.42 & -0.45 &  0.67 &  0.62 &  0.57 &  0.47 &  0.45 & -0.43 &  0.57 &  1.00
\end{pmatrix}
\end{equation*}
 
\begin{equation*}
\renewcommand{\arraystretch}{0.85}
\Sigma_Y = \scriptsize\begin{pmatrix}
 1.00 &  0.74 &  0.76 & -0.57 &  0.49 &  0.77 & -0.63 & -0.45 &  0.46 &  0.59 \\
 0.74 &  1.00 &  0.91 & -0.67 &  0.60 &  0.92 & -0.74 & -0.54 &  0.57 &  0.71 \\
 0.76 &  0.91 &  1.00 & -0.68 &  0.62 &  0.96 & -0.77 & -0.58 &  0.60 &  0.74 \\
-0.57 & -0.67 & -0.68 &  1.00 & -0.45 & -0.70 &  0.55 &  0.39 & -0.46 & -0.52 \\
 0.49 &  0.60 &  0.62 & -0.45 &  1.00 &  0.62 & -0.50 & -0.38 &  0.38 &  0.46 \\
 0.77 &  0.92 &  0.96 & -0.70 &  0.62 &  1.00 & -0.79 & -0.58 &  0.61 &  0.75 \\
-0.63 & -0.74 & -0.77 &  0.55 & -0.50 & -0.79 &  1.00 &  0.47 & -0.50 & -0.62 \\
-0.45 & -0.54 & -0.58 &  0.39 & -0.38 & -0.58 &  0.47 &  1.00 & -0.37 & -0.43 \\
 0.46 &  0.57 &  0.60 & -0.46 &  0.38 &  0.61 & -0.50 & -0.37 &  1.00 &  0.43 \\
 0.59 &  0.71 &  0.74 & -0.52 &  0.46 &  0.75 & -0.62 & -0.43 &  0.43 &  1.00
\end{pmatrix}
\end{equation*}
 
\begin{center}\small
\begin{tabular}{lccc}
\hline
Matrix & $\overline{|\rho|}$ & $\lambda_{\min}$ \\
\hline
$\Sigma_X$ & 0.578 & 0.042 \\
$\Sigma_Y$ & 0.600 & 0.035 \\
\hline
\end{tabular}
\end{center}
 
\subsection*{$I = \pm[0.67,\,1]$}
 
\begin{equation*}
\renewcommand{\arraystretch}{0.85}
\Sigma_X = \scriptsize\begin{pmatrix}
 1.00 &  0.77 & -0.88 & -0.86 & -0.83 & -0.79 & -0.73 &  0.77 & -0.84 & -0.75 \\
 0.77 &  1.00 & -0.87 & -0.86 & -0.83 & -0.79 & -0.73 &  0.76 & -0.83 & -0.76 \\
-0.88 & -0.87 &  1.00 &  0.98 &  0.94 &  0.90 &  0.84 & -0.88 &  0.95 &  0.87 \\
-0.86 & -0.86 &  0.98 &  1.00 &  0.93 &  0.89 &  0.83 & -0.86 &  0.94 &  0.85 \\
-0.83 & -0.83 &  0.94 &  0.93 &  1.00 &  0.85 &  0.80 & -0.83 &  0.90 &  0.83 \\
-0.79 & -0.79 &  0.90 &  0.89 &  0.85 &  1.00 &  0.76 & -0.79 &  0.85 &  0.78 \\
-0.73 & -0.73 &  0.84 &  0.83 &  0.80 &  0.76 &  1.00 & -0.75 &  0.80 &  0.75 \\
 0.77 &  0.76 & -0.88 & -0.86 & -0.83 & -0.79 & -0.75 &  1.00 & -0.83 & -0.75 \\
-0.84 & -0.83 &  0.95 &  0.94 &  0.90 &  0.85 &  0.80 & -0.83 &  1.00 &  0.83 \\
-0.75 & -0.76 &  0.87 &  0.85 &  0.83 &  0.78 &  0.75 & -0.75 &  0.83 &  1.00
\end{pmatrix}
\end{equation*}
 
\begin{equation*}
\renewcommand{\arraystretch}{0.85}
\Sigma_Y = \scriptsize\begin{pmatrix}
 1.00 &  0.92 &  0.92 & -0.84 &  0.79 &  0.93 & -0.87 & -0.76 &  0.78 &  0.85 \\
 0.92 &  1.00 &  0.97 & -0.88 &  0.84 &  0.98 & -0.92 & -0.80 &  0.83 &  0.90 \\
 0.92 &  0.97 &  1.00 & -0.88 &  0.84 &  0.99 & -0.93 & -0.82 &  0.84 &  0.91 \\
-0.84 & -0.88 & -0.88 &  1.00 & -0.76 & -0.89 &  0.83 &  0.72 & -0.76 & -0.81 \\
 0.79 &  0.84 &  0.84 & -0.76 &  1.00 &  0.85 & -0.79 & -0.70 &  0.71 &  0.77 \\
 0.93 &  0.98 &  0.99 & -0.89 &  0.85 &  1.00 & -0.93 & -0.82 &  0.84 &  0.92 \\
-0.87 & -0.92 & -0.93 &  0.83 & -0.79 & -0.93 &  1.00 &  0.77 & -0.79 & -0.86 \\
-0.76 & -0.80 & -0.82 &  0.72 & -0.70 & -0.82 &  0.77 &  1.00 & -0.70 & -0.75 \\
 0.78 &  0.83 &  0.84 & -0.76 &  0.71 &  0.84 & -0.79 & -0.70 &  1.00 &  0.76 \\
 0.85 &  0.90 &  0.91 & -0.81 &  0.77 &  0.92 & -0.86 & -0.75 &  0.76 &  1.00
\end{pmatrix}
\end{equation*}
 
\begin{center}\small
\begin{tabular}{lccc}
\hline
Matrix & $\overline{|\rho|}$ & $\lambda_{\min}$ \\
\hline
$\Sigma_X$ & 0.832 & 0.012 \\
$\Sigma_Y$ & 0.839 & 0.010 \\
\hline
\end{tabular}
\end{center}
 
\subsection*{$I = \pm[0.33,\,0.67]$}
 
\begin{equation*}
\renewcommand{\arraystretch}{0.85}
\Sigma_X = \scriptsize\begin{pmatrix}
 1.00 &  0.38 & -0.49 & -0.43 & -0.40 & -0.37 & -0.32 &  0.36 & -0.42 & -0.34 \\
 0.38 &  1.00 & -0.50 & -0.48 & -0.42 & -0.41 & -0.35 &  0.34 & -0.44 & -0.39 \\
-0.49 & -0.50 &  1.00 &  0.62 &  0.54 &  0.54 &  0.48 & -0.46 &  0.56 &  0.52 \\
-0.43 & -0.48 &  0.62 &  1.00 &  0.49 &  0.47 &  0.44 & -0.43 &  0.56 &  0.45 \\
-0.40 & -0.42 &  0.54 &  0.49 &  1.00 &  0.45 &  0.39 & -0.39 &  0.48 &  0.45 \\
-0.37 & -0.41 &  0.54 &  0.47 &  0.45 &  1.00 &  0.38 & -0.38 &  0.44 &  0.39 \\
-0.32 & -0.35 &  0.48 &  0.44 &  0.39 &  0.38 &  1.00 & -0.36 &  0.40 &  0.41 \\
 0.36 &  0.34 & -0.46 & -0.43 & -0.39 & -0.38 & -0.36 &  1.00 & -0.39 & -0.35 \\
-0.42 & -0.44 &  0.56 &  0.56 &  0.48 &  0.44 &  0.40 & -0.39 &  1.00 &  0.46 \\
-0.34 & -0.39 &  0.52 &  0.45 &  0.45 &  0.39 &  0.41 & -0.35 &  0.46 &  1.00
\end{pmatrix}
\end{equation*}
 
\begin{equation*}
\renewcommand{\arraystretch}{0.85}
\Sigma_Y = \scriptsize\begin{pmatrix}
 1.00 &  0.51 &  0.50 & -0.44 &  0.40 &  0.55 & -0.45 & -0.36 &  0.36 &  0.43 \\
 0.51 &  1.00 &  0.57 & -0.49 &  0.47 &  0.62 & -0.50 & -0.41 &  0.44 &  0.50 \\
 0.50 &  0.57 &  1.00 & -0.47 &  0.47 &  0.63 & -0.51 & -0.47 &  0.48 &  0.53 \\
-0.44 & -0.49 & -0.47 &  1.00 & -0.40 & -0.53 &  0.41 &  0.34 & -0.41 & -0.41 \\
 0.40 &  0.47 &  0.47 & -0.40 &  1.00 &  0.48 & -0.40 & -0.37 &  0.35 &  0.39 \\
 0.55 &  0.62 &  0.63 & -0.53 &  0.48 &  1.00 & -0.55 & -0.49 &  0.49 &  0.55 \\
-0.45 & -0.50 & -0.51 &  0.41 & -0.40 & -0.55 &  1.00 &  0.39 & -0.42 & -0.47 \\
-0.36 & -0.41 & -0.47 &  0.34 & -0.37 & -0.49 &  0.39 &  1.00 & -0.35 & -0.37 \\
 0.36 &  0.44 &  0.48 & -0.41 &  0.35 &  0.49 & -0.42 & -0.35 &  1.00 &  0.36 \\
 0.43 &  0.50 &  0.53 & -0.41 &  0.39 &  0.55 & -0.47 & -0.37 &  0.36 &  1.00
\end{pmatrix}
\end{equation*}
 
\begin{center}\small
\begin{tabular}{lccc}
\hline
Matrix & $\overline{|\rho|}$ & $\lambda_{\min}$ \\
\hline
$\Sigma_X$ & 0.434 & 0.355 \\
$\Sigma_Y$ & 0.455 & 0.347 \\
\hline
\end{tabular}
\end{center}
 
\subsection*{$ I =[0.33,\,0.67]$}
 
\begin{equation*}
\renewcommand{\arraystretch}{0.85}
\Sigma_X = \scriptsize\begin{pmatrix}
1.00 & 0.41 & 0.51 & 0.48 & 0.41 & 0.43 & 0.38 & 0.40 & 0.43 & 0.43 \\
0.41 & 1.00 & 0.51 & 0.44 & 0.40 & 0.40 & 0.37 & 0.38 & 0.42 & 0.39 \\
0.51 & 0.51 & 1.00 & 0.62 & 0.54 & 0.54 & 0.48 & 0.54 & 0.56 & 0.52 \\
0.48 & 0.44 & 0.62 & 1.00 & 0.49 & 0.47 & 0.44 & 0.48 & 0.56 & 0.45 \\
0.41 & 0.40 & 0.54 & 0.49 & 1.00 & 0.45 & 0.39 & 0.43 & 0.48 & 0.45 \\
0.43 & 0.40 & 0.54 & 0.47 & 0.45 & 1.00 & 0.38 & 0.42 & 0.44 & 0.39 \\
0.38 & 0.37 & 0.48 & 0.44 & 0.39 & 0.38 & 1.00 & 0.34 & 0.40 & 0.41 \\
0.40 & 0.38 & 0.54 & 0.48 & 0.43 & 0.42 & 0.34 & 1.00 & 0.46 & 0.42 \\
0.43 & 0.42 & 0.56 & 0.56 & 0.48 & 0.44 & 0.40 & 0.46 & 1.00 & 0.46 \\
0.43 & 0.39 & 0.52 & 0.45 & 0.45 & 0.39 & 0.41 & 0.42 & 0.46 & 1.00
\end{pmatrix}
\end{equation*}
 
\begin{equation*}
\renewcommand{\arraystretch}{0.85}
\Sigma_Y = \scriptsize\begin{pmatrix}
1.00 & 0.51 & 0.50 & 0.38 & 0.40 & 0.55 & 0.44 & 0.40 & 0.36 & 0.43 \\
0.51 & 1.00 & 0.57 & 0.45 & 0.47 & 0.62 & 0.52 & 0.46 & 0.44 & 0.50 \\
0.50 & 0.57 & 1.00 & 0.52 & 0.47 & 0.63 & 0.57 & 0.45 & 0.48 & 0.53 \\
0.38 & 0.45 & 0.52 & 1.00 & 0.38 & 0.50 & 0.42 & 0.34 & 0.34 & 0.45 \\
0.40 & 0.47 & 0.47 & 0.38 & 1.00 & 0.48 & 0.44 & 0.36 & 0.35 & 0.39 \\
0.55 & 0.62 & 0.63 & 0.50 & 0.48 & 1.00 & 0.57 & 0.47 & 0.49 & 0.55 \\
0.44 & 0.52 & 0.57 & 0.42 & 0.44 & 0.57 & 1.00 & 0.41 & 0.39 & 0.45 \\
0.40 & 0.46 & 0.45 & 0.34 & 0.36 & 0.47 & 0.41 & 1.00 & 0.34 & 0.42 \\
0.36 & 0.44 & 0.48 & 0.34 & 0.35 & 0.49 & 0.39 & 0.34 & 1.00 & 0.36 \\
0.43 & 0.50 & 0.53 & 0.45 & 0.39 & 0.55 & 0.45 & 0.42 & 0.36 & 1.00
\end{pmatrix}
\end{equation*}
 
\begin{center}\small
\begin{tabular}{lccc}
\hline
Matrix & $\overline{|\rho|}$ & $\lambda_{\min}$ \\
\hline
$\Sigma_X$ & 0.449 & 0.350 \\
$\Sigma_Y$ & 0.456 & 0.352 \\
\hline
\end{tabular}
\end{center}
 
\subsection*{$I = [0.67,\,1]$}
 
\begin{equation*}
\renewcommand{\arraystretch}{0.85}
\Sigma_X = \scriptsize\begin{pmatrix}
1.00 & 0.78 & 0.89 & 0.87 & 0.83 & 0.81 & 0.75 & 0.79 & 0.84 & 0.78 \\
0.78 & 1.00 & 0.88 & 0.86 & 0.82 & 0.79 & 0.74 & 0.77 & 0.83 & 0.76 \\
0.89 & 0.88 & 1.00 & 0.98 & 0.94 & 0.90 & 0.84 & 0.89 & 0.95 & 0.87 \\
0.87 & 0.86 & 0.98 & 1.00 & 0.93 & 0.89 & 0.83 & 0.87 & 0.94 & 0.85 \\
0.83 & 0.82 & 0.94 & 0.93 & 1.00 & 0.85 & 0.80 & 0.84 & 0.90 & 0.83 \\
0.81 & 0.79 & 0.90 & 0.89 & 0.85 & 1.00 & 0.76 & 0.80 & 0.85 & 0.78 \\
0.75 & 0.74 & 0.84 & 0.83 & 0.80 & 0.76 & 1.00 & 0.74 & 0.80 & 0.75 \\
0.79 & 0.77 & 0.89 & 0.87 & 0.84 & 0.80 & 0.74 & 1.00 & 0.85 & 0.78 \\
0.84 & 0.83 & 0.95 & 0.94 & 0.90 & 0.85 & 0.80 & 0.85 & 1.00 & 0.83 \\
0.78 & 0.76 & 0.87 & 0.85 & 0.83 & 0.78 & 0.75 & 0.78 & 0.83 & 1.00
\end{pmatrix}
\end{equation*}
 
\begin{equation*}
\renewcommand{\arraystretch}{0.85}
\Sigma_Y = \scriptsize\begin{pmatrix}
1.00 & 0.92 & 0.92 & 0.82 & 0.79 & 0.93 & 0.87 & 0.78 & 0.78 & 0.85 \\
0.92 & 1.00 & 0.97 & 0.87 & 0.84 & 0.98 & 0.92 & 0.82 & 0.83 & 0.90 \\
0.92 & 0.97 & 1.00 & 0.89 & 0.84 & 0.99 & 0.93 & 0.82 & 0.84 & 0.91 \\
0.82 & 0.87 & 0.89 & 1.00 & 0.76 & 0.89 & 0.83 & 0.73 & 0.74 & 0.83 \\
0.79 & 0.84 & 0.84 & 0.76 & 1.00 & 0.85 & 0.80 & 0.70 & 0.71 & 0.77 \\
0.93 & 0.98 & 0.99 & 0.89 & 0.85 & 1.00 & 0.93 & 0.82 & 0.84 & 0.92 \\
0.87 & 0.92 & 0.93 & 0.83 & 0.80 & 0.93 & 1.00 & 0.78 & 0.78 & 0.86 \\
0.78 & 0.82 & 0.82 & 0.73 & 0.70 & 0.82 & 0.78 & 1.00 & 0.69 & 0.77 \\
0.78 & 0.83 & 0.84 & 0.74 & 0.71 & 0.84 & 0.78 & 0.69 & 1.00 & 0.76 \\
0.85 & 0.90 & 0.91 & 0.83 & 0.77 & 0.92 & 0.86 & 0.77 & 0.76 & 1.00
\end{pmatrix}
\end{equation*}
 
\begin{center}\small
\begin{tabular}{lccc}
\hline
Matrix & $\overline{|\rho|}$ & $\lambda_{\min}$ \\
\hline
$\Sigma_X$ & 0.836 & 0.011 \\
$\Sigma_Y$ & 0.839 & 0.010 \\
\hline
\end{tabular}
\end{center}

\end{document}